\documentclass[12pt]{amsproc}
\usepackage{fullpage}
\usepackage{ulem}
\usepackage{subfigure}
\usepackage{graphicx}
\usepackage{enumerate}
\usepackage{amsmath,bm}
\usepackage{latexsym}
\usepackage{float}
\usepackage{latexsym}
\usepackage{amsmath}
\usepackage{amssymb}
\usepackage{amsfonts}
\usepackage{verbatim}
\usepackage{mathrsfs}
\usepackage{amsfonts}
\usepackage{colortbl,dcolumn}
\usepackage{amsmath}
\usepackage{amsfonts}
\usepackage{graphicx}
\usepackage{epstopdf}
\usepackage{algorithmic}

\newtheorem{tm}{Theorem}
\newtheorem{rk}{Remark}

\newtheorem{prop}{Proposition}
\newtheorem{lm}{Lemma}

\newtheorem{ex}{Example}

\newcommand{\E}{\mathbb E}

\newcommand{\bi}{\mathbf i}
\newcommand{\bs}{\mathbf s}

\newcommand{\<}{\langle}
\renewcommand{\>}{\rangle}

\begin{document}

\title{Quantifying the effect of random dispersion for logarithmic Schr\"odinger equation}

\author{Jianbo Cui}
\address{Department of Applied Mathematics, The Hong Kong Polytechnic University, Hung Hom, Hong Kong.}
\email{jianbo.cui@polyu.edu.hk}
\thanks{The research is partially supported by the Hong Kong Research Grant Council ECS grant 25302822, internal grants (P0039016, P0041274, P0045336) from Hong Kong Polytechnic University and the CAS AMSS-PolyU Joint Laboratory of Applied Mathematics. The research of the second author is supported by National Natural Science Foundation of China (No. 12101596, No. 12031020,
No.12171047).
}

\author{Liying Sun}
\address{Corresponding author. 
Academy for Multidisciplinary Studies, Capital Normal University,  Beijing 100048, China.}
\email{liyingsun@lsec.cc.ac.cn}

\begin{abstract}
This paper is concerned with the random effect of the noise dispersion for stochastic logarithmic Schr\"odinger equation emerged from the optical fibre with dispersion management. The well-posedness of the logarithmic Schr\"odinger equation with white noise dispersion is established via the regularization energy approximation and a spatial
{scaling}  property. 
For the small noise case, the effect of the noise dispersion is quantified by the proven large deviation principle under additional regularity assumptions on the initial datum. 
As an application, we show that for the regularized model, the exit from a neighborhood of the attractor of deterministic equation occurs on a sufficiently large time scale. 
Furthermore, the exit time and exit point in the small noise case, as well as the effect of large noise dispersion, is also discussed for the stochastic logarithmic Schr\"odinger equation. 
\end{abstract}



\maketitle

\section{Introduction}
The stochastic nonlinear Schr\"odinger equation describes the propagation of varying envelops of a wave packet in media with both weakly nonlinear and dispersive responses (\cite{SS99}), and has been applied in nonlinear optics, hydrodynamics, biology, crystals and Fermi-Pasta-Ulam chains of atoms. In nonlinear optics, the spantaneous emission noise in stochastic nonlinear Schr\"odinger equation appears since the amplifiers are placed along the fiber line to compensate for the loss caused by the weak damping in the fiber \cite{desurvire2002erbium}.
 Due to the inherent uncertainty on the amplified signal and quantum considerations, amplification
is intrinsically associated with small noise \cite{KLMM01W}.
In the context of crystal and  Fermi--Pasta--Ulam  chains of atoms, the noise accounts for thermal effects.

In this paper, we are interested in studying the random effect of the noise dispersion for the following stochastic  logarithmic Schr\"odinger equation (SLogSE)
\begin{align}\label{slogse}
du^{\epsilon}(t)&=\bi \sqrt{\epsilon} \Delta u^{\epsilon}(t)\circ dB(t)+\bi \lambda \log(|u^{\epsilon}(t)|^2)u^{\epsilon}(t)dt\\\nonumber
&\quad +\bi V[u^{\epsilon}(t)]u^{\epsilon}(t)dt-\alpha_1 u^{\epsilon}(t)dt
\end{align}
emerged from the optical fibre with dispersion management (\cite{Agra01b,Agra01a}). 
Here $u^{\epsilon}(0)=u_0$, $\lambda\in \mathbb R/\{0\}$ shows the force of nonlinear interaction of the logarithmic potential,   $\epsilon>0$ characterizes the intensity of noise dispersion and $\alpha_1>0$ is the weak damping coefficient over long distances. The Laplacian operator is defined on $\mathbb R^d$, $V[\cdot]$ represents a nonlocal interaction defined by $V[u](y)=\int_{\mathbb R^d} V(y-x)|u(x)|^2 dx$ for some function $V$, $B(\cdot)$ is the standard Brownian motion on a complete probability space $(\Omega,\mathcal F, \{\mathcal F_t\}_{t\ge 0}, \mathbb P)$,   and  the symbol ``$\circ$'' means that
the stochastic integral is understood in the Stratonovich sense. 

In the last two decades, there have been fruitful results and studies on stochastic nonlinear Schr\"odinger equations with polynomial and smooth nonlinearities driven by random forces. For instance, results on the well-posedness and the effect of a noise on the blow-up phenomenon {have} been proved in \cite{BRZ14,BHW22,CHS18b,BD99,DD02b,DD05,OO20} and references therein.  
Due to the loss of analytical expression of the solution, several numerical methods have been designed to simulate the behaviors of  stochastic nonlinear Schr\"odinger equations, such as {in}  \cite{BC23,CH16,CH17,CHL16b,CHLZ17,DD02a}, { just to name a few.} 
 Recently, more and more attention has been paid on stochastic nonlinear Schr\"odinger equations with random dispersion emerged from dispersion management (see, e.g., \cite{MR2652190, MR2832639,DR22,MR4167043,MR4278943}). In \cite{MR2652190}, it has been shown that stochastic nonlinear Schr\"odinger equation  with white noise dispersion is the limit of nonlinear Schr\"odinger equation with a scaling sequence of real-valued stationary random processes. Many researchers are also devoted into its numerical study (see, e.g., \cite{BDD15,MR4400428,MR3736655,MR3649275,MR4278943}).
In contrast, less result is known for the effect of white noise dispersion on stochastic nonlinear Schr\"odinger equations with logarithmic nonlinearities, which is one motivation of this current study. 

In determinstic case, it is known that the logarithmic nonlinearity makes the logarithmic Schr\"odinger equation,
\begin{align}\label{logse}
    \partial_t u(t)=\bi \Delta u(t) +\bi \lambda \log(|u(t)|^2)u(t),
\end{align}
unique among many nonlinear wave equations. It has been shown in \cite{CG18} that when $\lambda<0$, the modulus of the solution converges to a universal Gaussian profile by scaling in space via the dispersion rate. The idea of scaling in space and time plays an important role in studying \eqref{logse}. Inspired by such idea, one may formally define {$e^{-\bi l(t)|x|^2}v(t,l(t)x)=u(t,x)$} for $x\in \mathbb R^d$ and some positive (or negative) and continuous differentiable function $l$. 
{Then it follows that 
\begin{align*}
&\frac {d}{dt}e^{-\bi l(t)|x|^2}v(t,l(t)x)\\
=&\partial_t v(t, l(t) x)e^{-\bi l(t)|x|^2}
+\dot l(t)  \nabla_x v(t, l(t) x)\cdot x e^{-\bi l(t)|x|^2}-\bi \dot l(t) |x|^2e^{-\bi l(t)|x|^2}v(t,l(t)x)
\\
=&\bi \Delta (e^{-\bi l(t)|x|^2}v(t, l(t)x)) +\bi \lambda \log(|v(t, l(t) x)|^2)v(t,l(t)x)e^{-\bi l(t)|x|^2}.
\end{align*}
Direct calculations yield that $v$ satisfies   
\begin{align}\label{dis-det}
  \partial_t  v(t,x)&=\bi (l(t))^2\Delta v(t,x)+\bi \lambda \log(|v(t,x)|^2) v(t,x)\\\nonumber
  & +(4 (l(t))^2-\dot l(t))\nabla v\cdot x+(\bi \dot l(t) |x|^2-\bi 4 (l(t))^2|x|^2+2l(t))v(t,x).
\end{align}
Thus, under suitable conditions on $l(t)$, the properties of \eqref{logse} can be transformed into  those of \eqref{dis-det}  with the dispersion $l^2(t)$ via the above scaling technique. 
Unfortunately, this approach fails for studying \eqref{slogse} since $\epsilon \dot W (t)\neq (l(t))^2$ for any $l(t)$.}

Since \eqref{slogse} involves with the logarithmic nonlinearity, the techniques in \cite{MR2652190,MR2832639} could not be used directly to show its well-posedness. 
To deal with the singularity caused by the possible vacuum of logarithmic nonlinearity, we will exploit the idea and functional setting in \cite{CG18,CH23} where the regularized approximation {is} used.  Moreover, in section \ref{sec-m1} we present several subtle estimates on the uniform bound of the regularized approximations such that the global well-posedness of \eqref{slogse} could be established under certain assumptions on $V$. We also find an interesting scaling result in space for \eqref{slogse} and its regularization approximation. 
More precisely, denote the solution $u^{\epsilon,\delta}(t,x)$ of the following regularized SlogSE 
\begin{align}\label{reg-log-rands}
  du^{\epsilon,\delta}(t)&=\bi \Delta u^{\epsilon,\delta}(t)\circ \sqrt{\epsilon} dB(t)+\bi\lambda f_{\delta}(|u^{\epsilon,\delta}(t)|^2)u^{\epsilon,\delta}(t) dt\\\nonumber
  &\quad+\bi V[u^{\epsilon,\delta}(t)]u^{\epsilon,\delta}(t)dt-\alpha_1 u^{\epsilon,\delta}(t)dt
\end{align}
with $u^{\epsilon,\delta}(0)=u_0.$ Here $f_{\delta}$ is an approximation of the logarithmic function for $\delta>0$. For convenience, we also denote $f_{0}(\cdot)=\log(\cdot)$ and $u^{\epsilon,0}=u^{\epsilon}.$
Then defining $v^{\delta}(t,\epsilon^{\frac 14}x):=u^{\epsilon,\delta}(t,x)$, $v^{\delta}$ should satisfy
\begin{align}\label{reg-log-rand}
      dv^{\delta}(t)&=\bi \Delta v^{\delta}(t)\circ dB(t)+\bi\lambda f_{\delta}(|v^{\delta}(t)|^2)v^{\delta}(t) dt\\\nonumber 
      &\quad+\bi V[v^{\delta}(t)]v^{\delta}(t)dt- \alpha_1 v^{\delta}(t)dt,
\end{align}
where $v^{\delta}(0,x)=u_0(\epsilon^{-\frac 14}x).$
It can be seen that \eqref{reg-log-rands} is equivalent to \eqref{reg-log-rand} up to a scaling in space. 
As a by-product, we can have the following scaling properties, i.e., for $\alpha\ge 0$,
\begin{align*}
    \|u^{\epsilon,\delta}(t,\cdot)\|_{L_{\alpha}^2}=\|v^{\delta}(t, \epsilon^{\frac 14} (\cdot) )\|_{L_{\alpha}^2}=\|v^{\delta}(t,\cdot )\|_{L_{\alpha}^2},\; \\
    \|\nabla u^{\epsilon,\delta}(t,\cdot)\|=\|\nabla v^{\delta}(t, \epsilon^{\frac 14} (\cdot))\|\epsilon^{\frac 14}=\|\nabla v^{\delta}(t,\cdot)\|\epsilon^{\frac 14}
\end{align*}
for any $u^{\epsilon,\delta}(t)\in L_{\alpha}^2\cap  H^1$ and $t\ge 0$. Here $\|\cdot\|$ is the $L^2(\mathbb R^d)$-norm, $H^{\bs},\bs\ge 0,$ is the standard Sobolev space and the weighted Sobolev space $L_{\alpha}^2,\alpha\ge 0,$ is defined by 
$$L_{\alpha}^2:= \{z\in L^2| x\mapsto (1+|x|^2)^{\frac {\alpha}2} z(x) \in L^2 \}$$
with the norm $\|z\|_{L_{\alpha}^2}:=\|(1+|x|^2)^{\frac \alpha 2} z\|^2$.

Another motivation of this work lies on the study of 
the random effect of the noise dispersion for stochastic nonlinear Schr\"odinger equations. In the small noise case, the large deviation principle (LDP) of stochastic nonlinear Schr\"odinger equations with polynomial nonlinearities driven by additive and multiplicative noises have been addressed in \cite{Gau05b,Gau05}. Then the LDP are applied to studying  the asymptotic of    
 the time jitter in soliton transmission in \cite{DG08} and to quantify the exit time and exit points for the exit problem from a basin of attractor for weakly damped stochastic nonlinear Schr\"odinger equations in \cite{Gau08}. 
{When the LDP holds, the first order of the probability of rare event is that of Boltzmann theory and the square of the amplitude of the small noise acts as the temperature}. 
The rate functional of LDP generally characterizes the transition between two states and the exit from the basin of attractor of the deterministic system, which is also related to \textit{minimum action paths}. However, it is still unknown about the asymptotic behaviors and LDP of \eqref{slogse} due to the singularity on the possible vacuum. 

Thanks to the established well-posedness result of \eqref{slogse} and its regularization approximation \eqref{reg-log-rands}, we are able to partially answer the LDP problem of \eqref{slogse}.
By imposing additional regularity on the initial datum, in section \ref{sec-ldp} we derive the large deviation principle and its rate functional  for the regularization approximation of \eqref{slogse}, and then pass the limit to prove the LDP of the original system via the strong convergence property and Varadhan's contraction principle (see, e.g., \cite{MR2584458}). 
As an application, in section \ref{sec-exit}, we use the LDP to study the exit problem from a neighborhood of an asymptotically stable equilibrium point for the regularized equation \eqref{reg-log-rands} (or equivalently \eqref{reg-log-rand}), and prove that on a exponentially large time scale, the exit from the domains  of attraction for \eqref{reg-log-rands} occurs due to large fluctuations. We would like to remark that it is still hard to establish the LDP and quantify the exit problem for \eqref{slogse} directly when the considered support is $\mathcal C([0,T];H^1\cap L_{\alpha}^2).$ In section \ref{sec-fur}, we give further discussions on the related topics, including giving a rough result on the exiting points for \eqref{slogse}, proving an exponential estimate on a special rare event, and presenting the effect of the large dispersion, which may help understand the dynamics of \eqref{slogse}. 

\section{Logarithmic Schr\"odinger equation with random dispersion}
\label{sec-m1}
The rigorous derivation of \eqref{slogse} can be understood in the way of \cite{MR2652190,MR4278943}, once the well-posedness of \eqref{slogse} is established. Namely, 
\eqref{slogse} can be viewed as the limit of the following Schr\"odinger equation as $\sigma \to 0$,
\begin{align*}
   \frac {d z}{dt}=\bi\frac {\epsilon}{\sigma}m(\frac 1{\sigma^2})\Delta z+\bi V[z] z+\bi \lambda \log(|z|^2)z -\alpha_1z, 
\end{align*}
under suitable conditions on the centered stationary random process $m(\cdot)$. Thus, \eqref{slogse} also belongs to the category of stochastic Wasserstein Hamiltonian flows in the sense of \cite{CLZ2023}. 
Recall that the mild solution of \eqref{reg-log-rands} is defined by a stochastic process $u^{\epsilon,\delta}$ satisfying 
\begin{align*}
    u^{\epsilon,\delta}(t)=&S_{\sqrt{\epsilon}B}(t,0)u_0+\int_0^t \bi \lambda S_{\sqrt{\epsilon}B}(t,r) f_{\delta}(u^{\epsilon,\delta}(r))dr\\
    &+{\bf{i}} \int_0^t S_{\sqrt{\epsilon}B}(t,r) (V([u^{\epsilon,\delta}(r)])u^{\epsilon,\delta}(r))dr-\int_0^t \alpha_1 S_{\sqrt{\epsilon}B}(t,r) u^{\epsilon,\delta}(r) dr,\; a.s.
\end{align*}
Here $S_{\sqrt{\epsilon}B}(t,r)=e^{\bi \Delta \sqrt{\epsilon} (B(t)-B(r))}$, where $0\leq r\leq t$. 
Due to $B(0)$=0, we denote $S_{\sqrt{\epsilon}B}(t):=S_{\sqrt{\epsilon}B}(t,0).$
To prove the well-posedness of \eqref{slogse}, thanks to the spatial scaling technique, our idea is studying the regularization approximation \eqref{reg-log-rand} at first and then passing the limit on $\delta$.

Throughout this paper, we assume that $V\in \mathcal C^m_b(\mathbb R^d),$ for some {$m\in\mathbb N^+.$} 
We use $C$ to denote various constants which may change from line to line. 
Due to the Leibniz rule and integration by parts, similar to \cite{MR4400428}, one can verify that the mapping $\bi V[(\cdot)](\cdot)$ satisfies that for $h\in H^{m}, m=0,1,$
\begin{align}\label{growth}
\|\bi V[(v)]v\|_{H^{m}}&\le C(m,V) \|v\|^{2}\|v\|_{H^{m}},\\\label{first-order}
\| \partial_v (\bi V[(v)]v)\cdot h \|_{H^{m}}&\le C(m,V) (\|v\|^2 \|h\|_{H^m}+\|v\|\|h\|\|v\|_{ H^m}).
\end{align}
Assume that $f_{\delta}(|x|^2):=\log(\frac {\delta+|x|^2}{1+\delta|x|^2})$. Then \cite[Lemma 1]{CH23} yields that 
\begin{align}\label{growth-reg}
&\|\lambda \bi f_{\delta}(|v|^2)v\|\le |\lambda| |\log(\delta)| \|v\|,\\\label{lip-reg-var}
&|\<\lambda \bi f_{\delta}(|v|^2)v-\lambda \bi f_{\delta}(|w|^2)w, v-w\>|\le 4|\lambda|\|v-w\|^2,\\\label{loc-lip}
&\|\lambda \bi f_{\delta}(|v|^2)v-\lambda \bi f_{\delta}(|w|^2)w\|\le |\lambda|(|\log(\delta)|+2)\|v-w\|.
\end{align}
Here the complex inner product is defined by $\<w,z\>:=\Re \int_{\mathbb R^d} \bar w(x) z(x) dx.$
We suppose that the deterministic initial value $u_0\in L_{\alpha}^2\cap H^1$ for $\alpha\in [0,1].$ For convenience, we denote $H=H^0$ and $L^p:=L^p(\mathbb R^d;\mathbb C)$

In the following, we will frequently use the weighted Sobolev embedding inequality (\cite[Lemma 6]{CH23}), i.e., for $d\in \mathbb N^+, \eta\in (0,1)$ and $\alpha >\frac {d\eta}{2-2\eta}$, it holds that 
\begin{align}\label{wei-sob-ine}
   \|z\|_{L^{2-2\eta}}\le C\|z\|^{1-\frac {d\eta}{2\alpha(1-\eta)}}\|z\|^{\frac {d\eta}{2\alpha(1-\eta)}}_{L_{\alpha}^2},\;  z\in L_{\alpha}^2, 
\end{align}
and the Gagliardo--Nirenberg interpolation inequality, i.e., 
\begin{align}\label{GN-ine}
  \|z\|_{L^{2+2\eta'}} \le C\|z\|^{1-\frac {\eta'd}{2+2\eta'}}\|\nabla z\|^{\frac {\eta'd}{2+2\eta'}}, \; z\in L^{2+2\eta'}\cap H^1,
\end{align}
where $\frac {\eta'd}{2+2\eta'}\in (0,1)$ with $\eta'>0.$

Now we are in a position to present the well-posedness result of \eqref{reg-log-rands}.

\begin{tm}\label{tm-1}
Let $T>0$, $\delta\ge 0$ and $\alpha\in [0,1]$. There exists a unique mild solution $u^{\epsilon,\delta}\in\mathcal C([0,T];H),$ a.s., of \eqref{reg-log-rands} satisfying for any $p\ge 2,$
\begin{align}\label{moment-x1}
    &\E[\sup_{t\in [0,T]}\|u^{\epsilon,\delta}(t)\|^{p}_{H^1}]+\E[\sup_{t\in [0,T]}\|u^{\epsilon,\delta}(t)\|^{p}_{L_{\alpha}^2}]\\\nonumber 
    &\le C(T,\lambda,\alpha_1,p,\|u_0\|)(\|u_0\|_{H^1}^p+\|u_0\|_{L_{\alpha}^2}^p).
\end{align}
\end{tm}

\begin{proof}
Since $v^{\delta}(t,\epsilon^{\frac 14}x):=u^{\epsilon,\delta}(t,x),$ it suffices to prove the well-posedness of \eqref{reg-log-rand}. 
The proof combines the following two steps. 

\textbf{Step 1:} 
We first prove the well-posedness of \eqref{reg-log-rand}
for $\delta>0.$
Making use of the properties \eqref{growth}-\eqref{first-order}, \eqref{growth-reg} and \eqref{loc-lip}, standard procedures in \cite[section 2]{CH23} lead to the well-posedness of \eqref{reg-log-rand} in $ L^p(\Omega; \mathcal C([0,T];H))$ for $T>0$ and $p\ge 1$. 
Next we provide several useful uniform a priori estimates of $v^{\delta}$. 
To apply the It\^o formula rigorously, one needs to use suitable approximation procedures as in \cite{CH23}. Here we omit these tedious and standard arguments. 

From the It\^o formula and the chain rule it follows that 
\begin{align}\label{mass-con}
    \|v^{\delta}(t)\|=e^{-\alpha_1 t}\|v^{\delta}(0)\|,\; a.s.
\end{align}
In order to study the well-posedness for the case of $\delta=0$, we  consider the moment estimates under the $H^1$-norm and the weighted Sobolev {$L_{\alpha}^2$-norm} with $\alpha\in [0,1]$.
By the It\^o formula and the integration by parts,  it holds that  
{\small
\begin{align}\label{ene-evo}
    d \frac 12\|\nabla v^{\delta}(t)\|^2
    =&-\<\bi \Delta v^{\delta}(t),\Delta v^{\delta}(t)\> \circ dB(t)+\<\lambda \bi \log(|v^{\delta}(t)|^2)\nabla v^{\delta}(t),\nabla v^{\delta}(t)\>dt\\\nonumber 
    &+2\<\lambda \bi \frac {\Re(\bar v^{\delta}(t)\nabla v^{\delta}(t))}{|v^{\delta}(t)|^2} v^{\delta}(t), \nabla v^{\delta}(t)\>dt+\<\bi V([v^{\delta}(t)])\nabla v^{\delta}(t), \nabla v^{\delta}(t)\>dt\\\nonumber 
    &+\<\bi \nabla  V([v^{\delta}(t)]) v^{\delta}(t), \nabla v^{\delta}(t)\>dt
    -\alpha_1 \|\nabla v^{\delta}(t)\|^2dt.
\end{align}}Using the skew-symmetric property of the complex inner product, and applying
H\"older's, Young's and Gronwall's inequalities, together with \eqref{growth}, we obtain that for any $p\ge 1$,
\begin{align}\label{pri-h1}
  \E[\sup_{t\in [0,T]}\|\nabla v^{\delta}(t)\|^{2p}] 
  &\le \exp\Big(C(V,\lambda,\alpha_1,T,\|u_0\|)p\Big)\E[\|\nabla v^{\delta}(0)\|^{2p}].
\end{align}
By applying the It\^o's formula to the weighted Sobolev norm, we obtain that 
\begin{align*}
    d\frac 12\|v^{\delta}(t)\|_{L_\alpha^2}^2
    &=\<\bi \Delta v^{\delta}(t), (1+|x|^2)^{\alpha} v^{\delta}(t)\> dB(t)
    -\frac 12 \<(1+|x|^{2})^{\alpha} v^{\delta}(t),\Delta^2 v^{\delta}(t)\>dt\\
    &+\frac 12 \<(1+|x|^{2})^{\alpha} \Delta v^{\delta}(t),\Delta v^{\delta}(t)\>dt-\alpha_1 \| v^{\delta}(t)\|_{L_\alpha^2}^2dt.
\end{align*}
To proceed, recall that $\Delta u=\sum_{i=1}^d \partial_{x_i}^2 u$ and $\Delta^2 u=\sum_{i,j=1}^d \partial_{x_i}^2\partial_{x_j}^2 u.$
By using integration by parts, 
{\small
\begin{align*}
   &\quad-\frac 12 \<(1+|x|^2)^{\alpha} u,\sum_{i,j=1}^d \partial_{x_i}^2\partial_{x_j}^2 u\>=-\frac 12 \sum_{i,j=1}^d \<\partial_{x_j}^2 [(1+|x|^2)^{\alpha} u],\partial_{x_i}^2 u\>\\
   &=-\frac 12 \sum_{i,j=1}^d \<\partial_{x_j}[(1+|x|^2)^{\alpha-1}2\alpha x_j  u+(1+|x|^2)^{\alpha}\partial_{x_j}u],\partial_{x_i}^2 u\>\\
      &=-\frac 12 \sum_{i,j=1}^d \<(1+|x|^2)^{\alpha-1}2\alpha   u+
      (1+|x|^2)^{\alpha-1}2\alpha x_j  \partial_{x_j}u(t),\partial_{x_i}^2 u\>\\
      &\quad-\frac 12 \sum_{i,j=1}^d 
      \<(1+|x|^2)^{\alpha-1}2\alpha x_j \partial_{x_j}u+(1+|x|^2)^{\alpha}\partial_{x_j}^2u,\partial_{x_i}^2 u\>\\
      &=-\frac 12 \sum_{i,j=1}^d \<(1+|x|^2)^{\alpha-1}2\alpha   u,\partial_{x_i}^2 u\>-\sum_{i,j=1}^d 
      \<(1+|x|^2)^{\alpha-1}2\alpha x_j \partial_{x_j}u,\partial_{x_i}^2 u\>\\
      &\quad-\frac 12 \sum_{i,j=1}^d \<(1+|x|^2)^{\alpha}\partial_{x_j}^2u,\partial_{x_i}^2 u\>.
\end{align*}
}
Notice that  
 for $j\neq i$, it holds that
\begin{align*}
    &- 
      \<(1+|x|^2)^{\alpha-1}2\alpha x_j \partial_{x_j}u,\partial_{x_i}^2 u\>\\
      =&\<(1+|x|^2)^{\alpha-1}2\alpha x_j \partial_{x_j}\partial_{x_i}u,\partial_{x_i} u\>
     +
      \<(1+|x|^2)^{\alpha-2} 2(\alpha-1) x_i 2\alpha x_j \partial_{x_j}u,\partial_{x_i} u\>\\
      =&-\<(1+|x|^2)^{\alpha-1}2\alpha x_j \partial_{x_j}\partial_{x_i}\partial_{x_i}u, u\>
     -\<(1+|x|^2)^{\alpha-2}2(\alpha-1)x_ i2\alpha x_j \partial_{x_j}\partial_{x_i}u, u\>\\
      &+
      \<(1+|x|^2)^{\alpha-2} 2(\alpha-1) x_i 2\alpha x_j \partial_{x_j}u,\partial_{x_i} u\>. 
\end{align*}
Using the integration by parts again, we further obtain 
\begin{align*}
  &-\< \partial_{x_j}\partial_{x_i}\partial_{x_i}u, (1+|x|^2)^{\alpha-1}2\alpha x_j u\>\\
  =&\<\partial_{x_i}\partial_{x_i}u, (1+|x|^2)^{\alpha-1}2\alpha x_j \partial_{x_j} u\>+\<\partial_{x_i}\partial_{x_i}u, (1+|x|^2)^{\alpha-1}2\alpha  u\>\\
  &+\<\partial_{x_i}\partial_{x_i}u, (1+|x|^2)^{\alpha-2}2(\alpha-1)x_j 2\alpha x_j u\>.
\end{align*}
As a consequence, 
\begin{align*}
   &- 2
      \<(1+|x|^2)^{\alpha-1}2\alpha x_j \partial_{x_j}u,\partial_{x_i}^2 u\>\\
      =&
      \<\partial_{x_i}\partial_{x_i}u, (1+|x|^2)^{\alpha-1}2\alpha  u\>+\<\partial_{x_i}\partial_{x_i}u, (1+|x|^2)^{\alpha-2}2(\alpha-1)x_j 2\alpha x_j u\>\\
  &-\<(1+|x|^2)^{\alpha-2}2(\alpha-1)x_ i2\alpha x_j \partial_{x_j}\partial_{x_i}u, u\>\\
  &+
      \<(1+|x|^2)^{\alpha-2} 2(\alpha-1) x_i 2\alpha x_j \partial_{x_j}u,\partial_{x_i} u\>\\
    =&-\<\partial_{x_i} u,(1+|x|^2)^{\alpha-1}2\alpha \partial_{x_i} u+(1+|x|^2)^{\alpha-2}2(\alpha-1)2\alpha x_i u\>\\
    &-\<\partial_{x_i}u, (1+|x|^2)^{\alpha-2}2(\alpha-1)2\alpha x_j^2 \partial_{x_i} u\>\\
    &-\<\partial_{x_i}u, (1+|x|^2)^{\alpha-3}2(\alpha-2)2(\alpha-1)2\alpha x_j^2x_i  u\>\\
    &+\<\partial_{x_j}u, (1+|x|^2)^{\alpha-2}2(\alpha-1)2\alpha x_ix_j \partial_{x_i}u\>\\
    &+\<\partial_{x_j}u, (1+|x|^2)^{\alpha-3}2(\alpha-2)2(\alpha-1)2\alpha x_ix_jx_i u\>\\
    &+\<\partial_{x_j}u, (1+|x|^2)^{\alpha-2}2(\alpha-1)2\alpha x_j u\>\\
    &+ \<(1+|x|^2)^{\alpha-2} 2(\alpha-1) x_i 2\alpha x_j \partial_{x_j}u,\partial_{x_i} u\>.
\end{align*}
Notice that $(1+|x|^2)^{-1}|x|^{\zeta}\le C_{\zeta}$ for $\zeta\in [0,2].$ It holds that for $j\neq i$,
\begin{align}\label{int-part}
    |\<(1+|x|^2)^{\alpha-1}2\alpha x_j \partial_{x_j}u,\partial_{x_i}^2 u\>|
    &\le C(\|\partial_{x_i} u\|^2+\|u\|^2+\|\partial_{x_j}u\|^2).
\end{align}
Similarly,
one can verify \eqref{int-part}  for $j=i$.
By the BDG inequality and Gronwall's inequality, as well as \eqref{pri-h1}, we obtain that 
{\small
\begin{align}\label{pri-la2}
\E \Big[\sup_{t\in [0,T]}\|v^{\delta}(t)\|_{L_{\alpha}^2}^{2p}\Big]\le C_1(T,p)
\exp\Big(C(V,\lambda,\alpha_1,T,\|u_0\|)p\Big)\Big(\|v^{\delta}(0)\|_{H^1}^{2p}+\|v^{\delta}(0)\|_{L_{\alpha}^2}^{2p}\Big),
\end{align}
}
where $C_1(T,p)$ depends on $p$ quadratically.  The estimates \eqref{pri-h1} and \eqref{pri-la2} lead to \eqref{moment-x1} for $\delta>0$ by the spatial scaling property.

\textbf{Step 2:} 
We show that the limit of $v^{\delta}$ as $\delta \to 0$ is the unique mild solution of \eqref{slogse}. Notice that 
 $v^{\delta}(0,x)=u_0(\epsilon^{-\frac 14}x)$ is independent of $\delta.$ 
We take arbitrary small positive parameters $\delta_1,\delta_2>0$ such that $\delta_1\ge \delta_2$. 
Then considering the evolution of $\frac 12\|v^{\delta_1}(t)-v^{\delta_2}(t)\|^2$, it holds that
{\small
\begin{align*}
&\frac 12 \|v^{\delta_1}(t)-v^{\delta_2}(t)\|^2\\
=&\frac 12 \|v^{\delta_1}(0)-v^{\delta_2}(0)\|^2+\int_{0}^t\<v^{\delta_1}(s)-v^{\delta_2}(s), \bi \Delta (v^{\delta_1}(s)-v^{\delta_2}(s))\>\circ \sqrt{\epsilon} dB(s)\\
&+\int_0^t \<v^{\delta_1}(s)-v^{\delta_2}(s),\lambda \bi f_{\delta_1}(|v^{\delta_1}(s)|^2)v^{\delta_1}(s)-\lambda \bi f_{\delta_2}(|v^{\delta_2}(s)|^2)v^{\delta_2}(s)\> ds\\
&+\int_0^t \<v^{\delta_1}(s)-v^{\delta_2}(s),\bi V[v^{\delta_1}(s)]v^{\delta_1}(s)-\bi V[v^{\delta_2}(s)]v^{\delta_2}(s)\> ds\\
&-\int_0^t\alpha_1 \|v^{\delta_1}(s)-v^{\delta_2}(s)\|^2 ds\\
=&:\frac 12\|v^{\delta_1}(0)-v^{\delta_2}(0)\|^2 +I_1(t)+I_2(t)+I_3(t)+I_4(t).
\end{align*}
}Since $B(\cdot)$ is independent of the spatial variable, it follows that $I_1=0.$
The property of the logarithmic function and \eqref{lip-reg-var}  yield that 
{\small
\begin{align*}
	|I_2(t)|\le& C(\lambda)\int_0^t\|v^{\delta_1}(s)-v^{\delta_2}(s)\|^2ds\\
	&+\int_0^t|\<v^{\delta_1}(s)-v^{\delta_2}(s),\bi \lambda(f_{\delta_1}(|v^{\delta_1}(s)|^2)-f_{\delta_2}(|v^{\delta_2}(s)|^2))u^{\delta_2}(s) \>|ds\\
	\le&  C(\lambda)\int_0^t\|v^{\delta_1}(s)-v^{\delta_2}(s)\|^2ds\\
	&+|\lambda|\int_0^t\int_{\mathcal O} |v^{\delta_1}(s)-v^{\delta_2}(s)| |\log(\delta_1+|v^{\delta_1}(s)|^2)-\log(\delta_2+|v^{\delta_2}(s)|^2)||v^{\delta_2}(s)| dxds\\
	&+|\lambda|\int_0^t\int_{\mathcal O} |v^{\delta_1}(s)-v^{\delta_2}(s)| |\log(1+\delta_1|v^{\delta_1}(s)|^2)-\log(1+\delta_2|v^{\delta_2}(s)|^2)||v^{\delta_2}(s)| dxds\\
	\le& C(\lambda)\int_0^t\|v^{\delta_1}(s)-v^{\delta_2}(s)\|^2ds\\
	&+|\lambda|\int_0^t\int_{\mathcal O} |v^{\delta_1}(s)-v^{\delta_2}(s)| \log(1+\frac {(\delta_1-\delta_2)}{\delta_2+|v^{\delta}(s)|^2})|v^{\delta_2}(s)| dxds\\
	&+|\lambda|\int_0^t\int_{\mathcal O} |v^{\delta_1}(s)-v^{\delta_2}(s)| |\log(1+\frac {(\delta_1-\delta_2)|v^{\delta_2}|^2}{1+\delta_2|v^{\delta_2}(s)|^2})||v^{\delta_2}(s)| dxds\\
	=&:  C(\lambda)\int_0^t\|v^{\delta_1}(s)-v^{\delta_2}(s)\|^2ds+I_{21}+I_{22}.
\end{align*}}Applying Young's inequality, \eqref{wei-sob-ine} and \eqref{GN-ine}, it follows that for $\eta\in [0,\frac {2\alpha}{2\alpha+d})$, $\frac {\eta'd}{2\eta'+2}\in [0,1)$ and $\alpha\in (0,1],$
{\small
\begin{align*}
	I_{21}+I_{22}\le&  \int_0^{t} C|\lambda| \|v^{\delta_1}(s)-v^{\delta_2}(s)\|^2ds+ \int_0^{t} C|\lambda| \delta_1^{\eta}\|v^{\delta_2}(s)\|_{L^{2-2\eta}}^{2-2\eta}ds\\
	&+\int_0^{t} C |\lambda|\delta_1^{\eta'} \|v^{\delta_2}(s)\|_{L^{2+2\eta'}}^{2+2\eta'}ds\\
 \le& \int_0^{t} C|\lambda| \|v^{\delta_1}(s)-v^{\delta_2}(s)\|^2ds+\int_0^t C\delta_1^{\eta} \|v^{\delta_2}(s)\|_{L_{\alpha}^2}^{\frac {d\eta}{\alpha}} \|v^{\delta_2}(s)\|^{2-2\eta-\frac {d\eta}{\alpha}}ds\\
 &+\int_0^t C\delta_1^{\eta'} \|v^{\delta_2}(s)\|^{d\eta'} \|\nabla v^{\delta_2}(s)\|^{2\eta'+2-d\eta'}ds.
\end{align*}}Using 
\eqref{first-order} yields that 
\begin{align*}
	|I_3(t)|=C(m,V)\int_0^t(\|v^{\delta_1}(s)\|^2+\|v^{\delta_2}(s)\|^2)\|v^{\delta_1}(s)-v^{\delta_2}(s)\|^2 ds,
\end{align*}
which, together with 
the mass evolution law \eqref{mass-con}, implies that 
\begin{align*}
	|I_3(t)|\le C(m,V)(1+\|u_0\|^2)\int_0^t\|v^{\delta_1}(s)-v^{\delta_2}(s)\|^2 ds.
\end{align*} 
Based on the estimates of $I_1$-$I_3$ and using Gronwall's inequality, we have
\begin{align*}
	\|v^{\delta_1}-v^{\delta_2}\|_{L^{2p}(\Omega;\mathcal C([0,T];H))}&\le C({V,\lambda,T,p},\alpha_1,\|u_0\|)(\delta_1^{\frac \eta 2}+\delta_1^{\frac {\eta'}2}).
\end{align*}
Thus for any $\delta_n \to 0,$  $\{v^{\delta_n}\}_{n}$ forms a Cauchy sequence in $L^{2p}(\Omega;\mathcal C([0,T];H)).$  Then  standard arguments as in \cite[section 4]{CH23} show that there exists a unique limit of the Cauchy sequence $u^{\delta}$ as $\delta\to 0$, which is also the mild solution of \eqref{reg-log-rand} with $\delta=0$.
By the spatial scaling property between $v^{\delta}$ and $u^{\epsilon,\delta}$, we get \eqref{moment-x1} for $\delta=0$ and complete the proof.
\end{proof}

By an interpolation argument, one can also prove that for any $\delta_n \to 0,$  $\{v^{\delta_n}\}_{n}$ forms a Cauchy sequence in $L^{2p}(\Omega;\mathcal C([0,T];L^2_{\alpha}))\cap L^{2p}(\Omega;\mathcal C([0,T$ $]; H^{\bs}))$ with $\bs\in (0,1)$ and $\alpha\in (0,1)$.

\section{Large deviation principle}
\label{sec-ldp}
Thanks to Schilder's theorem (see e.g., \cite{DZ98lp}), it is known that $\sqrt{\epsilon}B(\cdot)$ satisfies the large deviation principle (LDP) with a good rate function 
$$I^W(g)=\begin{cases}
\frac {\|g\|_{\mathcal H_P}^2} 2, \qquad &g\in \mathcal H_P,\\
+\infty, \qquad  &g\notin \mathcal H_P,
\end{cases}$$
where $\mathcal H_P$ is the Cameron--Martin space of the standard Brownian motion defined by  $\{g\in W^{1,2}([0,T];\mathbb R), g(0)=0\}$ equipped with the norm $\|g\|_{\mathcal H_P}:=\sqrt{\int_0^T |h|^2dt},$ where $h=\dot g$. For the considered models, we have the following LDP result. 

\begin{prop}\label{prop-ldp}
Let $T>0,$ $\delta\ge 0$ and $\alpha \in [0,1].$
The family  $\{u^{\epsilon,\delta}\}_{\epsilon>0}$ satisfies an LDP of speed $\epsilon$
and good rate function \begin{align*}
  I_{u_0}^{\delta,T}(z)= \frac 12 \inf_{ L_{u_0}^{\delta} (g)=z} \|g\|_{\mathcal H_P}^2, \; z\in \mathcal C([0,T]; H),
\end{align*}
where $L_{u_0}^{\delta}$ is defined by the solution operator of the skeleton equation,
\begin{align*}
    L_{u_0}^{\delta}(g)(t)&=S_{g}(t)u_0+\int_0^t S_g(t,s)\bi\lambda f_{\delta}(|L_{u_0}^{\delta}(g)(s)|^2)L_{u_0}^{\delta}(g)(s)ds\\
    &+\int_0^t S_g(t,s) \bi V[L_{u_0}^{\delta}(g)(s)] L_{u_0}^{\delta}(g)(s)ds-\int_0^t S_g(t,s) \alpha_1 L_{u_0}^{\delta}(g)(s) ds
\end{align*}
with $S_g(t,s)=\exp(\bi\Delta(g(t)-g(s))).$     
\end{prop}

\begin{proof}
    For the case that $\delta>0$, by the contraction principle, it suffices to prove the continuity of $L_{u_0}^{\delta}$, which is obtain in Lemma \ref{lm-con-ske}.  
    Notice that $u^{\epsilon,\delta}$ is proven to be an exponentially good approximation of $u^{\delta}$ in Lemma \ref{lm-exp-good}.
By \cite[Theorem 4.2.23]{DZ98lp}, to verify that $I^{0,T}_{u_0}$ the good rate function for \eqref{reg-log-rands} with $\delta=0$, 
it suffices to show that  
\begin{align}\label{ex-good}
\lim_{\delta \to 0} \sup_{\{g:\|g\|_{\mathcal H_P\le R}\}}\|L_{u_0}^{\delta}(g)-L_{u_0}^0(g)\|=0
\end{align}
for every positive $R<+\infty$. 
Since $g\in \mathcal H_P$, by similar arguments as in the proof of Theorem \ref{tm-1},  it is not hard to show that there exists a unique solution for  
\begin{align*}
  dL_{u_0}^{\delta}(g)(t)&=\bi \Delta L_{u_0}^{\delta}(g)(t)dg(t)+\bi \lambda f_{\delta}(|L_{u_0}^{\delta}(g)(t)|^2)L_{u_0}^{\delta}(g)(t)dt\\
  &\quad+\bi V[L_{u_0}^{\delta}(g)(t)]L_{u_0}^{\delta}(g)(t)dt-\alpha_1L_{u_0}^{\delta}(g)(t) dt, 
\end{align*}
satisfying  
\begin{align}\label{apri-lug}
   &\|L_{u_0}^{\delta}(g)(t)\|=e^{-\alpha_1 t}\|u_0\|,\; \|\nabla L_{u_0}^{\delta}(g)(t)\|\le C({V,\lambda,T},\alpha_1) \|\nabla u_0\|,\\\label{apri-lug1}
   &\|L_{u_0}^{\delta}(g)(t)\|_{L_{\alpha}^2}\le C({V,\lambda,T},\alpha_1)(\|u_0\|_{L_{\alpha}^2}+\| u_0\|_{H^1}\|g\|_{W^{1,2}([0,T];\mathbb R)}).
\end{align}
The properties \eqref{first-order} and \eqref{lip-reg-var}  yield that 
\begin{align*}
   &d\|L_{u_0}^{\delta}(g)(t)-L_{u_0}^{0}(g)(t)\|^2\\
   \le& 
     C(\lambda)\| (f_{\delta}(|L_{u_0}^{0}(g)(t)|^2)L_{u_0}^{0}(g)(t)-\log(|L_{u_0}^{0}(g)(t)|^2))L_{u_0}^{0}(g)(t)\|
     ^2 dt\\
   &+C({V,\lambda},\alpha_1)(1+\|L_{u_0}^{\delta}(g)(t)\|^{2}+\|L_{u_0}^{0}(g)(t)\|^{2})\|L_{u_0}^{\delta}(g)(t)-L_{u_0}^{0}(g)(t)\|^2dt.
\end{align*}
Combining the above estimates with \eqref{err-con} in the proof of Lemma \ref{lm-exp-good},  \eqref{ex-good} follows.
\end{proof}

\begin{lm}\label{lm-con-ske}
Let $T>0,$ $\delta\ge 0$, $\alpha \in [0,1]$ and $g\in \mathcal H_{P}$. The operator $L_{u_0}^{\delta}$ is continuous from $\mathcal C_0([0,T];\mathbb R)$ to $\mathcal C([0,T]; H).$
\end{lm}

\begin{proof}
For any $g_1,g_2\in \mathcal H_P$, using the mild formulations of $L_{u_0}^{\delta}(g_1)$ and $L_{u_0}^{\delta}(g_2)$, it holds that 
\begin{align*}
    &\|L_{u_0}^{\delta}(g_1)(t)-L_{u_0}^{\delta}(g_2)(t)\|\le \|S_{g_1}(t)u_0-S_{g_2}(t)u_0\|
    \\
    &+\Big\|\int_0^t (S_{g_1}(t,s)-S_{g_2}(t,s))\Big(\bi\lambda f_{\delta}(|L_{u_0}^{\delta}(g_1)(s)|^2)L_{u_0}^{\delta}(g_1)(s)-\alpha_1 L_{u_0}^{\delta}(g_1)(s)\Big) ds\Big\|\\
    &+\Big\|\int_0^t (S_{g_1}(t,s)-S_{g_2}(t,s))\bi V[L_{u_0}^{\delta}(g_1)(s)]L_{u_0}^{\delta}(g_1)(s) ds\Big\|\\
    &+\Big\|\int_0^t S_{g_2}(t,s)\bi\lambda \Big(f_{\delta}(|L_{u_0}^{\delta}(g_1)(s)|^2) L_{u_0}^{\delta}(g_1)(s)-f_{\delta}(|L_{u_0}^{\delta}(g_2)(s)|^2)L_{u_0}^{\delta}(g_2)(s)\\
    &\quad +\alpha_1 L_{u_0}^{\delta}(g_2)(s) -\alpha_1 L_{u_0}^{\delta}(g_1)(s)\Big)ds\Big \|\\
    &+\Big\|\int_0^t S_{g_2}(t,s)\bi \Big(V[L_{u_0}^{\delta}(g_1)(s)]L_{u_0}^{\delta}(g_1)(s) -V[L_{u_0}^{\delta}(g_2)(s)]L_{u_0}^{\delta}(g_2)(s)\Big)ds\Big\|.
\end{align*}
On the one hand, for any $u_0 \in H^1$, by Fourier's transformation, it holds that  for $\|g_1-g_2\|_{\mathcal C([0,T];\mathbb R)}\to 0,$
\begin{align*}
    \|(S_{g_1}(t)-S_{g_2}(t))u_0\|
    &=\|(I-S_{g_2-g_1}(t))u_0\|\le C\|u_0\|_{H^1} |g_2(t)-g_1(t)|^{\frac 12}
    \to 0.
\end{align*}
As a consequence, by \eqref{apri-lug}-\eqref{apri-lug1}, \eqref{growth} and \eqref{growth-reg}, we have that for $\|g_1-g_2\|_{\mathcal C([0,T];\mathbb R)}\to 0,$
\begin{align*}
  &\quad\|\int_0^t (S_{g_1}(t,s)-S_{g_2}(t,s)) \Big(\bi\lambda f_{\delta}(|L_x^{\delta}(g_1)(s)|^2)L_x^{\delta}(g_1)(s)-\alpha_1 L_x^{\delta}(g_1)(s)\Big)ds\|
  \\
  &\quad+\|\int_0^t (S_{g_1}(t,s)-S_{g_2}(t,s))\bi V[L_{u_0}^{\delta}(g_1)(s)]L_{u_0}^{\delta}(g_1)(s) ds\|\\
  &\le C({\lambda},\alpha_1)\int_0^t \|(S_{g_1}(t,s)-S_{g_2}(t,s))\|_{\mathcal L( H; H^1)} \Big(\|V[L_{u_0}^{\delta}(g_1)(s)]L_{u_0}^{\delta}(g_1)(s)\|_{ H^1}\\
  &\quad +\|f_{\delta}(|L_{u_0}^{\delta}(g_1)(s)|^2)L_{u_0}^{\delta}(g_1)(s)\|_{ H^1}+\|L_{u_0}^{\delta}(g_1)(s)\|_{ H^1}\Big)ds\\
  &\le C({T,V,\lambda,|\log(\delta)|})(1+\|u_0\|^{2})\int_0^t |g_1(t)-g_1(s)-g_2(t)+g_1(s)|^{\frac 12}
  \|u_0\|_{ H^1}ds
  \to 0.
\end{align*}
Making use of properties 
\eqref{first-order} and  \eqref{loc-lip}, it follows that 
\begin{align*}
    &\Big\|\int_0^t S_{g_2}(t,s)\bi\lambda \Big(f_{\delta}(|L_{u_0}^{\delta}(g_1)(s)|^2) L_{u_0}^{\delta}(g_1)(s)-f_{\delta}(|L_{u_0}^{\delta}(g_2)(s)|^2)L_{u_0}^{\delta}(g_2)(s)\\
    &\quad +\alpha_1 L_{u_0}^{\delta}(g_2)(s) -\alpha_1 L_x^{\delta}(g_1)(s)\Big)ds\Big \|\\
    &+\Big\|\int_0^t S_{g_2}(t,s)\bi \Big(V[L_{u_0}^{\delta}(g_1)(s)]L_{u_0}^{\delta}(g_1)(s) -V[L_{u_0}^{\delta}(g_2)(s)]L_{u_0}^{\delta}(g_2)(s)\Big)ds\Big\|\\
    \le& C({V,\lambda,|\log(\delta)|},T,\alpha_1)(1+\|u_0\|^{2})\int_0^t \|L_{u_0}^{\delta}(g_1)(s)-L_{u_0}^{\delta}(g_2)(s)\| ds.
\end{align*}
Then standard arguments yield that 
\begin{align*}
    &\|L_{u_0}^{\delta}(g_1)-L_{u_0}^{\delta}(g_2)\|_{C([0,T];H)}
    \\
    \le &C({V,\lambda,|\log(\delta)|},T,\alpha_1)(1+\|u_0\|^{2})(1+\|u_0\|_{H^1})|g_1-g_2|_{C([0,T];\mathbb R)},
\end{align*} 
which completes the proof.
\end{proof}

\begin{lm}\label{lm-exp-good}
Let $T>0$, $\delta\ge 0$ and $\alpha\in [0,1]$. The sequence $\{u^{\epsilon,\delta}\}_{\delta>0}$ is an exponentially good approximation of $u^{\epsilon}$, i.e., for any $\delta_1>0,$
\begin{align*}
 \lim_{\delta \to 0}\lim\sup_{\epsilon \to 0}\epsilon \log \mathbb P(\|u^{\epsilon,\delta}-u^{\epsilon}\|_{C([0,T]; H)}>\delta_1)=-\infty.
\end{align*}
\end{lm}

\begin{proof}
We prove that $u^{\epsilon,\delta}$ is an exponentially good approximation of $u^{\epsilon}$ by the strong convergence property under the $L^p(\Omega)$-norm. By the chain rule, the proprieties \eqref{lip-reg-var} and \eqref{first-order}, it holds that 
\begin{align*}
  d\|u^{\epsilon,\delta}-u^{\epsilon}\|^2
  =&2\<\bi \Delta (u^{\epsilon,\delta}-u^{\epsilon}),u^{\epsilon,\delta}-u^{\epsilon}\>\circ dB(t)\\
  &+2\<\bi\lambda ( f_{\delta}(|u^{\epsilon,\delta}|^2)u^{\epsilon,\delta}-\log(|u^{\epsilon}|^2)u^{\epsilon}),u^{\epsilon,\delta}-u^{\epsilon}\>dt\\
  &+2\<\bi (V[u^{\epsilon,\delta}]u^{\epsilon,\delta}-V[u^{\epsilon}]u^{\epsilon}),u^{\epsilon,\delta}-u^{\epsilon}\>\\
  &-2\alpha_1\|u^{\epsilon,\delta}(t)-u^{\epsilon}(t)\|^2dt\\
  \le& C(\lambda)\| (f_{\delta}(|u^{\epsilon}(t)|^2)u^{\epsilon}(t)-\log(|u^{\epsilon}(t)|^2))u^{\epsilon}(t)\|
     ^2 dt\\
   &+C(V,\lambda.\alpha_1)(1+\|u^{\epsilon,\delta}\|^{2}+\|u^{\epsilon}\|^{2})\|u^{\epsilon,\delta}(t)-u^{\epsilon}(t)\|^2dt.
\end{align*}
According to \eqref{GN-ine} and \eqref{wei-sob-ine}, it follows that 
for $\eta\in [0,\frac {2\alpha}{2\alpha+d})$, $\frac {\eta'd}{2\eta'+2}\in [0,1)$ and $\alpha\in (0,1]$,
\begin{align}\label{err-con}
   &\quad\| (f_{\delta}(|u_1|^2)u_1-\log(|u_1|^2))u_1\|
     ^2\\\nonumber
  &\le  C(\|\log(1+\delta|u_1|^2)u_1\|^2+\|\log(1+\frac {\delta}{|u_1^2|})u_1\|^2)\\\nonumber
  &\le C(\delta^{\eta}+\delta^{\eta'})(\|u_1\|_{L_{\alpha}^2}^{\frac {d\eta}{\alpha}}\|u_1\|^{2-2\eta-\frac {d\eta}{\alpha}}+\|u_1\|^{d\eta'}\|\nabla u_1\|^{2\eta'+2-d\eta'}),
\end{align}
where $u_1\in L_{\alpha}^2\cap H^1.$
By using \eqref{err-con} and Gronwall's inequality, we get that 
\begin{align*}
   \|u^{\epsilon,\delta}(t)-u^{\epsilon}(t)\|^2
   &\le e^{C({V,\lambda},\alpha_1)(1+\|u_0\|^2)T}(\delta^{\eta}+\delta^{\eta'})(1+\|u_0\|^{2-2\eta-\frac {d\eta}{\alpha}}+\|u_0\|^{d\eta'})\\
   &\quad \int_0^T (\|u^{\epsilon}\|_{L_{\alpha}^2}^{\frac {d\eta}{\alpha}}+\|\nabla u^{\epsilon}\|^{2\eta'+2-d\eta'}) dt.
\end{align*}
Then according to the priori estimate of $u^{\epsilon}$ and applying Chebyshev's inequality, it holds that
\begin{align*}
  &\epsilon \log \mathbb P(\|u^{\epsilon,\delta}-u^{\epsilon}\|_{\mathcal C([0,T]; H)}\ge \delta_1)\\
  \le&   \epsilon \log \mathbb P(\|u^{\epsilon,\delta}-u^{\epsilon}\|_{\mathcal C([0,T]; H)}^{\frac 1\epsilon}\ge \delta_1^{\frac 1{\epsilon}})\\
  \le&  \epsilon \log \frac {C(V,\lambda,T,\alpha_1,u_0))^{\frac 1{\epsilon}}\E[\Big(\int_0^T (\|u^{\epsilon}\|_{L_{\alpha}^2}^{\frac {d\eta}{\alpha}}+\|\nabla u^{\epsilon}\|^{2\eta'+2-d\eta'}) dt\Big)^{\frac 1 {\epsilon}}]\delta^{\max(\frac {\eta}{2\epsilon},\frac {\eta'}{2\epsilon})}} {\delta_1^{\frac 1\epsilon}}\\
  \le& C_1(V,\lambda,T,\alpha_1,u_0)+\log(\frac {\delta^{\frac 12\max(\eta,\eta')}}{\delta_1}).
\end{align*}
Here we have use the estimate 
\begin{align*}
\E[\Big(\|u^{\epsilon}(t)\|_{L_{\alpha}^2}^{\frac {d\eta}{\alpha}}+\|\nabla u^{\epsilon}(t)\|^{2\eta'+2-d\eta'}\Big)^{\frac 1{\epsilon}}]&\le \frac 1{\epsilon}   C_2(V,\lambda,T,\alpha_1,u_0)^{\frac 1{\epsilon}}.
\end{align*}
which can be obtained by It\^o's formula similar to \eqref{pri-h1} and \eqref{pri-la2}.
Then letting $\delta\to 0,$ we have complete the proof for the exponentially good approximation property.

\end{proof}

\begin{rk}\label{rk-exp-good}
From the above analysis, it can be seen that the strongly convergent approximation with explicit convergence rate and finite moment bounds under $L^p(\Omega)$-norm for sufficiently large $p\ge 1$ is also an exponentially good approximation for stochastic (partial) differential equation.  By a classical approach in \cite[Chapter 5]{DZ98lp}, 
one could also prove the uniform LDP for \eqref{reg-log-rands} with $\delta\ge 0$, for any compact set $K \subset L_\alpha^2 \cap H^1$, for any $A\in \mathcal B(\mathcal C([0,T]; H))$,
\begin{align}\label{ldp}
 &-\sup_{u_0\in K}\inf_{w\in Int(A)} I^{\delta,T}_{u_0}(w)\le \lim\inf_{\epsilon\to0}\epsilon \log\inf_{u_0\in K} \mathbb P(u^{\epsilon}\in A)\\\nonumber
 &\le \lim\sup_{\epsilon \to 0}\epsilon \log\sup_{u_0\in K}\mathbb P(u^{\epsilon}\in A)\le -\inf_{w\in \bar A, u_0\in K}I_{u_0}^{\delta,T}(w). 
\end{align}
\end{rk}

In the following, we present a useful LDP result of \eqref{reg-log-rands} with $\delta>0$ when the domain of $I_{u_0}^{\delta}$ is considered on $\mathcal C([0,T];\mathcal X_1)$, whose proof is in the appendix. Here $\mathcal X_1:=H^1\cap L_1^2.$ 

For any $a\ge 0$, denote $K_T^{u_0,\delta}(a)=(I_{u_0}^{\delta,T})^{-1}([0,a]),$ i.e., 
$$K_T^{u_0,\delta}(a)=\Big\{y\in \mathcal C([0,T];\mathcal X_1)| y=L_{u_0}^{\delta}(g),\frac 12 \int_0^T|h|^2dt\le a, h=\dot g \Big\}.$$

\begin{tm}\label{ldp-chi1}
    Let $\delta>0,T>0$, $u_0\in \mathcal X_1$ and $d\le 2$. For every $a,\rho,\kappa,\gamma$ positive,
    \begin{enumerate}
        \item[(i)] there exits $\epsilon_0>0$ such that for every $\epsilon \in (0,\epsilon_0)$ and $\|u_0\|_{\mathcal X_1}\le \rho$ and $\widetilde a\in (0,a],$
        \begin{align*}
            \mathbb P(d_{\mathcal C([0,T];\mathcal X_1)}(u^{\epsilon,\delta,u_0},K_{T}^{u_0}(\widetilde a))\ge \gamma)<\exp\Big(-\frac {\widetilde a -\kappa}{\epsilon}\Big),
        \end{align*}
         \item[(ii)] there exists $\epsilon_0>0$ such that for every $\epsilon \in (0,\epsilon_0)$ and $\|u_0\|_{\mathcal X_1}\le \rho$ and $w\in K_{T}^{u_0,\delta}(a),$
           \begin{align*}
            \mathbb P(\|u^{\epsilon,\delta,u_0}-w\|_{\mathcal C([0,T];\mathcal X_1)}<\gamma)>\exp\Big(-\frac {I_{u_0}^{\delta,T}(w) -\kappa}{\epsilon}\Big).
        \end{align*}
    \end{enumerate}
\end{tm}

The assumption $d\le 2$ is technical for deriving a priori estimates in the Sobolev space of higher order.
It should be remarked that our current approach does not give the LDP result of Theorem \ref{ldp-chi1} in $\mathcal C([0,T];\mathcal X_1)$ for the case $\delta=0$ due to the strong singularity of the logarithmic nonlinearity in the Sobolev space of higher order. For further discussion, we refer to section \ref{sec-fur}.

\section{Application: exit from a basin of attraction}
\label{sec-exit}
This section is devoted to the first exit time  from a neighborhood of an asymptotically stable equilibrium point for \eqref{reg-log-rands} with $\delta>0$ and $\alpha_1>0$. For simplicity, we may assume that $V=0.$ We refer to \cite{DZ98lp,MR722136} for more backgrounds for the exit problems and its connection to LDP.

Considering the following infinite-dimensional ordinary differential equation with weakly damping effect,
\begin{align}\label{det-log-ode}
 dw=\bi\lambda f_{\delta}(|w|^2)wdt-\alpha_1 w dt,   
\end{align} it holds that 
\begin{align*}
    &|w(t,x)|=e^{-\alpha_1 t}|w(0,x)|,\\
    &w(t,x)=\exp\Big(-\alpha_1 t+\int_0^t\bi f_{\delta}(|w(0,x)|^2e^{-2\alpha_1 s})ds\Big) w(0,x).
\end{align*}
Thus, $0$ is the unique attractor of the above equation in $H$.  Note that in the $H$-topology, the exit problems are not interesting since 
\begin{align*}
    \|w^{\epsilon,\delta}(t)\|\le e^{-\alpha_1 t} \|w^{\epsilon,\delta}(0)\|.
\end{align*}
Similarly, one may consider the topology in $L_{\alpha}^2$ since $0$ is also the attractor of the considered equation.
For simplicity, we can take $\alpha=1$ and obtain 
\begin{align*}
  \|w(t)\|_{L_1^2}\le e^{-\alpha_1 t} \|w(0)\|_{L_1^2}.  
\end{align*}
However, $0$ may be not the attractor in $H^\bs, \bs>0.$
For example, when $\bs=1$, by Young's inequality,
\begin{align*}
d\|\nabla w(t)\|^2&=\<\bi \lambda f'_{\delta}(|w(t)|^2)Re(\bar w(t)\nabla w(t))w(t),\nabla w(t) \>dt-2\alpha_1\|\nabla w(t)\|^2\\
&\le (-2\alpha_1+4|\lambda|) \|\nabla w(t)\|^2dt,
\end{align*}
which
implies that if $\alpha_1>2|\lambda|$, $0$ is also the attractor in $H^{1}.$


Define a new norm $\|\cdot\|_{\mathcal X_{\bs}}$ defined by $\sqrt{\|\cdot\|_{H^{
\bs
}}^2+\|\cdot \|_{L_{\bs}^2}^2}$ with $\bs\ge 0$. We will consider the exit problem under $\|\cdot\|_{\mathcal X_1}$-norm. Consider an open bounded domain $D$ containing $0$ in the interior of $\mathcal X_{1}$ such that $D\subset B_{R}$ for a large enough $R>0.$ The above analysis indicates that 
$D$ is invariant under the deterministic flow of \eqref{det-log-ode} if  $\alpha_1>2|\lambda|$.

Define $\tau^{\epsilon,\delta,u_0}=\inf\{t\ge 0| u^{\epsilon,\delta,u_0}(t)\in D^{c}\}$ the first exit time of the regularized SlogS equation from $D.$ Similar to \cite{Gau08}, we introduce $$\overline{M^{\delta}}=\inf\{I_0^{\delta,T}(y): y(T)\in (\overline{D})^c, T>0\}.$$
For any sufficient positive $\rho>0,$ set
$$M_{\rho}^{\delta}=\inf\{I_{u_0}^{\delta,T}(y):\|u_0\|_{\mathcal X_1}\le \rho, y(T)\in (D_{-\rho})^c,T>0\}$$
with $D_{-\rho}:=D\backslash \mathcal  N^0(\partial D,\rho)$ where $\mathcal  N^0(\partial D,\rho)$ is the open neighbourhood of $\partial D$ with the distance $\rho$. Here $\partial D$ is the boundary of $D$ in $\mathcal X_1.$
Define 
$\underline {M^{\delta}}=\lim\limits_{\rho\to 0}M_{\rho}^{\delta}$.
It can be seen that $\underline{M^{\delta}}\le \overline{M^{\delta}}.$
Below, we shall prove that the lower bound of $\underline{M^{\delta}}$ is strictly larger than $0$ thanks to special structure of the skeleton equation.

\begin{lm}\label{lm-0}
Let $\delta>0$, $\alpha_1>2|\lambda|$ and $\alpha=1$. Then it holds that 
$0<\inf\limits_{\delta>0 } \underline{M^{\delta}}\le \inf\limits_{\delta> 0}\overline{M^{\delta}}.$
\end{lm}
\begin{proof}
 Let $d(0,\partial D)>0$ denote the distance between $0$ and $\partial D.$
 Choose $\rho$ small enough such that
 $B_{\rho}^0\subset D$ and that the distance between $B_{\rho}^0$ and $(D_{-\rho})^c$ is larger than $\frac 12d(0,\partial D)$. 
 By studying the functional $\mathcal H(L_{u_0}^{\delta}(g)):=\frac 12 \|L_{u_0}^{\delta}(g)\|^2_{\mathcal X_1}$, it follows that
\begin{align*}
    &d\mathcal H(L_{u_0}^{\delta}(g)(t))\\
    =&\<2x L_{u_0}^{\delta}(g)(t),\bi \nabla  L_{u_0}^{\delta}(g)(t)\> h(t) dt 
    -\alpha_1 \mathcal H(L_{u_0}^{\delta}(g)(t))dt\\
    &+\<\nabla L_{u_0}^{\delta}(g)(t), \bi\lambda  f'_{\delta}(|L_{u_0}^{\delta}(g)(t)|^2)Re(\overline{L_{u_0}^{\delta}(g)}(t)\nabla L_{u_0}^{\delta}(g)(t))L_{u_0}^{\delta}(g)(t)\>dt\\
    \le& 2\|L_{u_0}^{\delta}(g)(t)\|_{L_1^2}\|\nabla  L_{u_0}^{\delta}(g)(t)\| h(t)dt
    -\alpha_1 \mathcal H(L_{u_0}^{\delta}(g)(t))dt+2|\lambda|
  \|\nabla L_{u_0}^{\delta}(g)(t)\|^2dt,
\end{align*}
where $h=\dot g.$
Using the Duhamel formula, we obtain that 
\begin{align*}
    \mathcal H(L_{u_0}^{\delta}(g)(T))-e^{-(\alpha_1-2|\lambda|+2\delta )T}\mathcal H(L_{u_0}^{\delta}(g)(0))
    \le \int_0^T e^{-(\alpha_1-2|\lambda|+2\delta)(T-s)}2R^2|h(s)| ds.
\end{align*}
As a consequence, if $L_{u_0}(g)(T)\in (D_{-\rho})^c$,  it holds that 
\begin{align*}
  \frac {d^2(0,\partial D)-4d(0,\partial D)\rho}{16} \le R^2  \sqrt{\int_0^Te^{-2(\alpha_1-2|\lambda|)(T-s)}ds}\|h\|_{L^2([0,T];\mathbb R)}.
\end{align*}
Taking $\rho$ sufficient small leads to 
\begin{align*}
    \frac {d^2(0,\partial D)}{32 R^2\sqrt{2\alpha_1-4|\lambda|}}\le \|h\|_{L^2([0,T];\mathbb R)},
\end{align*}
 thus the desired results follow. 
\end{proof}

We are not able to prove $\underline{M^{\delta}}=\overline{M^{\delta}}$ due to loss of the approximate controllability  at this current study. The argument to prove the approximate controllability of the original system (i.e., $\delta=0$) is hard since the nonlinearity is not locally Lipschitz and the Schr\"odinger group relies on the control. 
In the ideal case, it is expected that $$\lim_{\delta \to 0} M^{\delta}=\inf_{v\in \partial D}U(0,v),$$
with the quasi-potential defined by
$$U(u_0,u_1):=\inf\{I_{u_0}^{0,T}(y)| y\in \mathcal C([0,T],\mathcal X_1), y(0)=u_0, y(T)=u_1, T>0\}.$$
Define $\sigma_{\rho}^{\epsilon,\delta,u_0}:=\inf\{t\ge 0: u^{\epsilon,\delta,u_0}(t)\in B_{\rho}^0\cup D^c\}.$

\begin{lm}\label{lm-1}
Let $\delta>0$, $\alpha_1>2|\lambda|$ and $\alpha=1$. 
For every $\rho$ small enough and $L$ positive with $B_{\rho}^0\subset D$, there exists $T>0$ and $\epsilon_0$ such that for every $u_0\in D$ and $\epsilon \in (0,\epsilon_0),$
  \begin{align*}
      \mathbb P(\sigma_{\rho}^{\epsilon,\delta,u_0}>T)\le \exp(-\frac L{\epsilon}).
  \end{align*}
\end{lm}

\begin{proof}
The proof for the case that $u_0$ belongs to $B_{\rho}^0$ is straightforward. Below we focus on the case that $u_0\in D \backslash B_{\rho}^0.$
Since $D$ is uniformly attracted to zero by the deterministic flows \eqref{det-log-ode}, there exists a positive time $T_1>0$ such that for every $u_1$ in $\mathcal N^0(D \backslash B_{\rho}^0,\frac \rho 8)$ and $t\ge T_1$, $L_{u_1}^{\delta}(0)(t)\in B^0_{\frac \rho 8}.$
Notice that 
\begin{align*}
    \sup_{u_1\in \mathcal N^0(D\backslash B_{\rho}^0,\frac \rho 8)}\|L_{u_1}^{\delta}(0)\|_{\mathcal C([0,T_1],\mathcal X_1)}\le C(T_1,|\lambda|,|\alpha_1|,R).
\end{align*}
Next, we prove that for sufficient large $T\ge T_1,$ 
\begin{align}\label{claim1}
\mathcal T_{\rho}&=\{y\in \mathcal C([0,T];\mathcal X_1)| \forall t\in [0,T], y(t)\in \mathcal N^0(D \backslash B_{\rho}^0,\frac \rho 8)\}\\\nonumber
&\subset (K_{T}^{u_0,\delta}(2L))^c.
\end{align}
It suffices to consider $y=L_{u_0}^{\delta}(g)\in \mathcal T_{\rho}.$ Indeed, we have that
\begin{align*}
    \|L_{u_0}^{\delta}(g)-L_{u_0}^{\delta}(0)\|_{\mathcal C([0,T_1];\mathcal X_1)}&\ge \|L_{u_0}^{\delta}(g)(T_1)-L_{u_0}^{\delta}(0)(T_1)\|_{\mathcal X_1}\ge \frac 34 \rho.
\end{align*}
On the other hand, there exists small $t_1>0$ such that 
\begin{align*}
&\|L_{u_0}^{\delta}(g)-L_{u_0}^{\delta}(0)\|_{\mathcal C([0,t_1];\mathcal X_1)}\\
&\le C(R',\delta,R) |\int_0^{t_1}|h(s)|ds|^{\frac {1}2} +\alpha_1 t_1 \|L_{u_0^{R'}}^{\delta}(g)-L_{u_0^{R'}}^{\delta}(0)\|_{\mathcal C([0,t_1];\mathcal X_1)}\\
&+|\lambda||\log(\delta)| t_1 \|L_{u_0}^{\delta}(g)-L_{u_0}^{\delta}(0)\|_{\mathcal C([0,t_1];\mathcal X_1)}\\
&\le C(R',\delta,R)t_1^{\frac 14} |\int_0^{t_1}|h(s)|^2ds|^{\frac {1}4} +\alpha t_1 \|L_{u_0}^{\delta}(g)-L_{u_0}^{\delta}(0)\|_{\mathcal C([0,t_1];\mathcal X_1)}\\
&+|\lambda||\log(\delta)| t_1 \|L_{u_0}^{\delta}(g)-L_{u_0}^{\delta}(0)\|_{\mathcal C([0,t_1];\mathcal X_1)}.
\end{align*}
Letting $t_1$ sufficient small such that $|\lambda||\log(\delta)| t_1 +\alpha t_1 \le \frac 12$ and $C(R',\delta,R)t_1^{\frac 14}\le \frac 12$, we obtain that 
\begin{align*}
  \|L_{u_0}^{\delta}(g)-L_{u_0}^{\delta}(0)\|_{\mathcal C([0,t_1];\mathcal X_1)}&\le |\int_0^{t_1}|h(s)|^2ds|^{\frac {1}4}.
\end{align*}
Then by iterating the above estimate for each small interval $[kt_1,(k+1)t_1]$ for $k=1,\cdots, [T_1/t_1 -1],$ ($[\cdot]$ is the floor function), it follows that 
\begin{align*}
  \|L_{u_0}^{\delta}(g)-L_{u_0}^{\delta}(0)\|_{\mathcal C([0,T_1];\mathcal X_1)}&\le 2^{[T_1/t_1]+1}\|h\|_{L^2([0,T_1];\mathbb R)}^{\frac 12}.
\end{align*}
Thus, we conclude that 
\begin{align*}
   \frac 1{2^{4[T/t_1]+5}} (\frac {3}4 \rho)^4 \le  \frac 12 \|h\|_{L^2([0,T_1];\mathbb R)}^2. 
\end{align*}
Similarly, for any $[T_1,2T_1]$, by the inverse triangle inequality and the fact that $0$ is an attractor of the deterministic flow,
one has that 
\begin{align*}
&\|L_{L_{u_0}^{\delta}(g)(T_1)}^{\delta}(g)-L_{{L_{u_0}^{\delta}(g)(T_1)}}^{\delta}(0)\|_{\mathcal C([0,T_1];\mathcal X_1)}\ge \frac {1}2\rho.
\end{align*}
This also implies that 
\begin{align*}
\frac 12 \|h\|_{L^2([T_1,2T_1];\mathbb R)}^2\ge  \frac 1{2^{4[T/t_1]+5}} (\frac {1}2 \rho)^4=:M''.
\end{align*}
Therefore, we obtain that 
\begin{align*}
    \frac 12 \|h\|^2_{L^2([0,2T_1];\mathbb R)}\ge 2M''.
\end{align*}
For any $T>0$, there exists positive number $j$  such that $T>j T_1$, iterating the above arguments, one has $\frac 12 \|h\|^2_{L^2([0,jT_1];\mathbb R)}\ge jM''.$ 
To obtain \eqref{claim1}, we will take $jM''>2L.$

Finally, we are able to show the desired result since 
\begin{align*}
    \mathbb P(\sigma_{\rho}^{\epsilon,\delta,u_0}>T)
    &=\mathbb P(\forall t\in [0,T], u^{\epsilon,\delta,u_0}\in D\backslash B_{\rho}^0)\\
    &=\mathbb P(d_{\mathcal C([0,T];\mathcal X_1)}(u^{\epsilon,\delta,u_0},\mathcal T_{\rho}^c)> \frac {\rho}8)\\
    &\le \mathbb P(d_{\mathcal C([0,T];\mathcal X_1)}(u^{\epsilon,\delta,u_0},K_T^{u_0}(2L))\ge  \frac {\rho}8).
\end{align*} 
By Theorem \ref{ldp-chi1} $(i)$, it follows that for small $\epsilon>0$ and $\rho\le R$, 
\begin{align*}
    &\mathbb P(d_{\mathcal C([0,T];\mathcal X_1)}(u^{\epsilon,\delta,u_0},K_T^{u_0}(2L))\ge \frac {\rho} {8})
    \le 
    e^{-\frac {L}\epsilon}.
\end{align*}
\end{proof}

The below lemma indicates that at the time $\sigma_{\rho}^{\epsilon,u_0,\delta}$, the probability that $u^{\epsilon,u_0,\delta}$ escapes $D$ is at most exponentially small.

\begin{lm}\label{lm-2}
Let $\delta>0$, $\alpha_1>2|\lambda|$ and $\alpha=1$.  
For every $\rho>0$ such that $B_{\rho}^0\subset D$ and $u_0\in D$, there exists $L>0$ such that 
\begin{align*}
\lim\sup_{\epsilon\to 0}\epsilon\log \mathbb P\Big(u^{\epsilon,\delta,u_0}(\sigma_{\rho}^{\epsilon,\delta,u_0})\in \partial D\Big)\le -L.
\end{align*}
\end{lm}
\begin{proof}
When $u_0\in B_{\rho}^0$, the desired bound is trivial. Thus we only deal with the case that $u_0\in D\backslash B_{\rho}^0$. 
Since zero is the attractor of the deterministic flow, define $T=\inf\{t\ge 0| L_{u_0}^{\delta}(0)(t)\in B^0_{\frac \rho 2}\}$ and it holds that 
 \begin{align*}
     \mathbb P\Big(u^{\epsilon,\delta,u_0}(\sigma_{\rho}^{\epsilon,\delta,u_0})\in \partial D\Big)
     \le \mathbb P\Big( \|u^{\epsilon,\delta,u_0}-L_{u_0}^{\delta}(0)\|_{\mathcal C([0,T];\mathcal X_1)}\ge \min(\frac \rho2,\frac {d(0,\partial D)}2)\Big).
 \end{align*}
 By the LDP, we get 
 \begin{align*}
\lim\sup_{\epsilon \to 0}\epsilon \log \mathbb P\Big( \|u^{\epsilon,\delta,u_0}-L_{u_0}^{\delta}(0)\|_{\mathcal C([0,T];\mathcal X_1)}\ge \min(\frac \rho2,\frac {d(0,\partial D)}2)\Big)\le -L,
 \end{align*}
 where 
\begin{align*}
L=\inf_{\|L_{u_0}^{\delta}(0)-L_{u_0}^{\delta}(g)\|_{\mathcal C([0,T];\mathcal X_1)}\ge \min(\frac \rho2,\frac {d(0,\partial D)}2)} \|h\|^2_{L^2([0,T];\mathbb R)}>0.
\end{align*}
\end{proof}

\begin{lm}\label{lm-4}
Let $\delta>0$, $\alpha_1>2|\lambda|$ and $\alpha=1$.      For every $\rho>0$ and $L>0$ such that $B_{2\rho}^0\subset D$, there exists $T<+\infty$ such that 
  \begin{align*}
     \lim\sup_{\epsilon \to 0} \epsilon \log \sup_{u_0\in S_{\rho}^0} \mathbb P\Big(\sup_{t\in [0,T]}(\mathcal H(u^{\epsilon,\delta,u_0})-\mathcal H(u_0)) \ge \frac 3 {2} \rho^2\Big) \le -L,
  \end{align*}
  where $S_{\rho}^0$ is the sphere in $\mathcal X_1.$
\end{lm}

\begin{proof}
The evolution of $\mathcal H(u^{\epsilon,\delta, u_0})$ yields that 
\begin{align*}
 &\mathcal H(u^{\epsilon,\delta, u_0}(t))- \mathcal H(u_0)\\
 =&\sqrt{\epsilon}\int_0^t\<\bi 2x u^{\epsilon,\delta, u_0},\nabla u^{\epsilon,\delta, u_0}\>dW(s)
 -\frac 12\int_0^t \epsilon\<(1+|x|^2)u^{\epsilon,\delta, u_0},\Delta^2 u^{\epsilon,\delta, u_0}\>ds\\
 &+\frac 12\int_0^t \epsilon \<(1+|x|^2)\Delta u^{\epsilon,\delta, u_0}, \Delta u^{\epsilon,\delta, u_0}\>ds
 -\alpha_1 \int_0^t \mathcal H(u^{\epsilon,\delta, u_0}(s))ds\\
 &+\int_0^t \bi 2 \lambda \<\partial_x f_{\delta}(|u^{\epsilon,\delta, u_0}|^2) Re(\overline{u^{\epsilon,\delta, u_0}}\nabla u^{\epsilon,\delta, u_0}) u^{\epsilon,\delta, u_0}, \nabla  u^{\epsilon,\delta, u_0}\>ds.
\end{align*}
By integration by parts, Holder's and Young's inequalities, it is enough to show for $\epsilon\in(0,\epsilon_0)$ with small $\epsilon_0>0$ and for small $T(\rho,L)<1$, 
\begin{align*}
     \lim\sup_{\epsilon \to 0} \epsilon \log \sup_{u_0\in S_{\rho}^0} \mathbb P\Big(\sup_{t\in [0,T(\rho,L)]} |Z(t)| \ge \frac {\rho^2}{\sqrt{\epsilon}} \Big) \le -L,
  \end{align*}
  where $Z(t)=\int_0^t\<2x u^{\epsilon,\delta,u_0}, \bi \nabla u^{\epsilon,\delta,u_0}\>dW(t)$.
By using \cite[Proposition 4.31]{DZ14}, it holds that for any $b>0$, 
\begin{align*}
\mathbb P(\sup_{[0,T(\rho,L)]}|Z(t)|\ge \frac {\rho^2}{\sqrt{\epsilon}})
&\le \frac {b\epsilon}{\rho^4}
+\mathbb P( \int_0^{T(\rho,L)}4\|xu^{\epsilon,\delta,u_0}\|^2\|\nabla u^{\epsilon,\delta,u_0}\|^2 ds\ge b ).
\end{align*}
Note that by Chebyshev's inequality and the moment bound of $H(u^{\epsilon,\delta,u_0})$, it holds that for $q=\frac 1{\epsilon},$ 
\begin{align*}
    \mathbb P(4 \int_0^{T(\rho,L)}\|xu^{\epsilon,\delta,u_0}\|^2\|\nabla u^{\epsilon,\delta,u_0}\|^2 ds\ge b )
    &\le \frac {4^qT_{\rho,L}^{q}\sup\limits_{t\in [0,T_{\rho,L}]}\E[H(u^{\epsilon,\delta,u_0})^{2q}]}{b^{q}}\\
    &\le \frac {4^qT_{\rho,L}^{q}
{C(\lambda_1,\alpha_1,T_{\rho,L})^{q}q\rho^{4q}}}{b^{q}}.
\end{align*}
Taking $b=(\frac 1{L'})^{\frac {1}\epsilon}$ for a sufficient large $L'$, it follows that 
\begin{align*}
&\epsilon \log \sup_{u_0\in S_{\rho}^0} \mathbb P( 4\int_0^{T(\rho,L)}\|xu^{\epsilon,\delta,u_0}\|^2\|\nabla u^{\epsilon,\delta,u_0}\|^2 ds\ge b ) \\
\le& \log(4T_{\rho,L})-\log(b)+ \log(\rho^4)+\epsilon\log(\frac 1\epsilon) +\log(C(\lambda_1,\alpha_1,T_{\rho,L})).
\end{align*}
As a consequence, 
\begin{align*}
 &\lim\sup_{\epsilon \to 0}\epsilon \log \sup_{u_0\in S_{\rho}^0} \mathbb P\Big(\sup_{t\in [0,T(\rho,L)]} |Z(t)| \ge \frac {\rho^2}{\sqrt{\epsilon}} \Big)\\
 \le& \max\Big(\log(\frac 1{L'})-\log(\rho^4), \log(4T_{\rho,L})-\log(\frac 1{L'})+\log(\rho^4)+\log(C(\lambda_1,\alpha_1,T_{\rho,L}))\Big).
\end{align*}
To complete the proof, one just takes $T_{\rho,L}\le \frac {1}{4C(\lambda_1,\alpha_1,T_{\rho,L})\rho^4L'}$ with $L'$ sufficient large.

\end{proof}

Now we are able to present the theorem to characterize the first exist time from a given domain $D$ for the regularized   problem ($\delta>0$). Its proof is put in the appendix. 

\begin{prop}\label{tm-exit}
Let $\delta>0$, $\alpha_1>2|\lambda|$ and $\alpha=1$. 
For every $u_0\in D$ and small $\kappa>0$, there exists positive $L$ such that 
\begin{align}\label{tm-exit1}
    \lim\sup_{\epsilon \to 0} \epsilon \log \mathbb P(\tau^{\epsilon,\delta,u_0}\notin (e^{\frac {\underline{M^{\delta}}-\kappa}{\epsilon}},e^{\frac {\overline{M^{\delta}}+\kappa }\epsilon }))\le -L, 
\end{align}
and for every $u_0\in D$,
\begin{align}\label{tm-exit2}
\underline{M^{\delta}}\le\lim\inf_{\epsilon \to 0}\epsilon \log \E [\tau^{\epsilon,\delta,u_0}]\le \lim\sup_{\epsilon \to 0}\epsilon \log \E [\tau^{\epsilon,\delta,u_0}]\le  \overline{M^{\delta}}. 
\end{align}
Furthermore, for every small $\kappa>0$, there exists $L>0$ such that 
\begin{align}\label{tm-exit3}
   \lim\sup_{\epsilon\to 0}\epsilon \log\sup_{u_0\in D}\mathbb P\Big(\tau^{\epsilon,\delta,u_0}\ge e^{\frac {\overline{M^{\delta}}+\kappa}\epsilon}\Big)\le -L,\\ \label{tm-exit4}
   \lim\sup_{\epsilon \to 0}\log \sup_{u_0\in D}\E[\tau^{\epsilon,\delta,u_0}]\le \overline{M^{\delta}}.
\end{align}
\end{prop}

\section{Further discussions}
\label{sec-fur}
For the limit model, i.e., \eqref{reg-log-rands}  with $\delta=0$, there are still a lot of unclear parts on the random effect of the noise. Below we briefly discuss about some potential and interesting aspects in studying the effect of stochastic dispersion on the logarithmic Schr\"odinger equation.  

\subsection{Exit time and exit points} 

Similar to \cite{Gau08}, one can also formally characterize the exit points of the regularised problem.  However, it is still difficult to derive a rigorous result on the exit points of \eqref{reg-log-rands} with $\delta=0$ and $V=0$  since  Proposition \ref{tm-exit} may fail if $\delta \to 0$. Thanks to the special structures \eqref{mass-con}  and \eqref{ene-evo} when $\alpha_1>2|\lambda|$, we  have the following rough result on the exit points. 

\begin{prop}\label{prop-2}
Let $\delta\ge 0$, $\alpha_1>2|\lambda|$ and $\alpha=1$. For \eqref{reg-log-rands}, the exit from an open bounded domain $D$ containing $0$ in the interior of $\mathcal X_{1}$  appears in $\mathcal X_1$ if and only if the exit  appears in $L_1^2.$
\end{prop}
The proof is straightforward by noticing that $$\|u^{\epsilon,\delta}(t)\|\le e^{-\alpha_1 t}\|u_0\|, \; \|\nabla u^{\epsilon,\delta}(t)\| \le  e^{-(\alpha_1-2|\lambda|) t}\|\nabla u_0\|.$$  
Furthermore, we have the following finding on  the exit time of \eqref{reg-log-rands} with $\delta=0$ based on the exponential integrability property.  

\begin{prop}\label{prop-3}
Let $\delta=0$, $\alpha_1>2|\lambda|$ and $\alpha=1$. Assume that $R>\||x|u_0\|$. Then for any $T>0,$ it holds that 
\begin{align*}
   \limsup_{\epsilon \to 0} \epsilon \log \mathbb P(\tau^{\epsilon,u_0}\le T)=- {(R^2-\||x| u_0\|^2)},
\end{align*}
Where $\tau^{\epsilon,u_0}$ is the exit time from the ball $\mathcal B_R^0$ in $L_1^2$.
\end{prop}

\begin{proof}
According to Proposition \ref{prop-2}, it suffices to analyze the 
exit problem in $L_1^2.$ 
 For any $T>0$, it follows that 
\begin{align*}
    \mathbb P(\tau^{\epsilon,u_0}\le T)\le \mathbb P(\sup_{t\in [0,T]} \||x| u^{\epsilon,u_0}\| \ge R).
\end{align*}
By applying \cite[Lemma 3.1]{CHL16b}, it holds that for any $p\ge 0,$ 
{\small
\begin{align*}
    &\exp(\frac {p\||x| u^{\epsilon,u_0}(t)\|^2}{\epsilon})\\
    \le&
     \exp(\frac {p \||x| u_0\|^2}{\epsilon})
     +\int_0^t (C(\lambda,p)\|u^{\epsilon,u_0}(s)\|^2_{H^1}-\frac {\alpha_1 p}{\epsilon} \||x| u^{\epsilon,u_0}(t)\|^2)\exp(\frac {p\||x| u^{\epsilon,u_0}(s)\|^2}{\epsilon})ds
    \\
    &+\int_0^t p\exp(\frac {p\||x| u^{\epsilon,u_0}(t)\|^2}{\epsilon})\<\bi \nabla u^{\epsilon,u_0}, 2x u^{\epsilon,u_0}\>\sqrt{\epsilon}dB(t).
\end{align*}
}
Combining the above estimate with the BDG inequality,  we can obtain that 
\begin{align*}
    &\E[\sup_{s\in [0,T]}\exp(\frac {p \||x| u^{\epsilon,u_0}(t)\|^2}{\epsilon})]\\
    \le& \exp (\frac {p\||x| u_0\|^2}{\epsilon})\exp\Big(\int_0^t C(\lambda,p,\alpha_1)e^{-2(\alpha_1-2|\lambda|)s}ds \|u_0\|_{H^1}^2\Big).
\end{align*}
Taking $p=1$, by Chebyshev's inequality, we have that 
\begin{align*}
    &\mathbb P(\sup_{t\in [0,T]} \||x| u^{\epsilon,u_0}\| \ge R)\\
    &\le \exp (-\frac {R^2}{\epsilon})\E[\sup_{s\in [0,T]}\exp(\frac {\||x| u^{\epsilon,u_0}(t)\|^2}{\epsilon})]\\
    &\le  \exp (-\frac {R^2-\||x| u_0\|^2}{\epsilon}) \exp\Big( C(\lambda,\alpha_1) \frac 1{2(\alpha_1-2|\lambda|)} \|u_0\|_{H^1}^2\Big).
\end{align*}
If $R^2>\||x| u_0\|^2,$ then it holds that for any $T>0,$
\begin{align*}
   \limsup_{\epsilon \to 0} \epsilon \log \mathbb P(\tau^{\epsilon,u_0}\le T)=- {(R^2-\||x| u_0\|^2)}.
\end{align*}

\end{proof}

In order to predict more delicate result on the exit time and exit points for the limit model, we aim to design some structure-preserving methods to simulate its asymptotic  behaviors in the future.
\subsection{Effect on the large dispersion} Another interesting problem lies on the effect of large random dispersion for \eqref{reg-log-rands} with $\delta=0$, i.e., $\epsilon \to +\infty.$ For example, let us assume that $V=0$ and $\alpha_1=0$. In order to avoid  confusions, we denote the solution of \eqref{reg-log-rands} with $\delta=0$ by $X^{\epsilon}$ for large eough $\epsilon>0.$ By the analysis in section \ref{sec-m1}, it is not hard to obtain 
\begin{align*}
  \|X^{\epsilon}(t)\|&=\|u_0\|, \; \|X^{\epsilon}(t)\|_{ H^1}\le e^{C(\lambda,\|u_0\|)t}\|u_0\|_{ H^1},\\
\|X^{\epsilon}(t)\|_{L^p(\Omega;L_{\alpha}^2)}&\le e^{C(\lambda,\|u_0\|,p) t}\epsilon \|u_0\|_{H^1}+C(p) \|u_0\|_{L_{\alpha}^2}.
\end{align*}
Since the decaying estimate of $S_{\sqrt{\epsilon} B}(t,s)$ holds \cite{MR2652190}, i.e.,  for $\frac 1p+\frac 1{p'}=1$ and $s<t,$
\begin{align*}
\|S_B(t,s)z\|_{L^p}&=\Big\|\frac 1{4 \pi \bi (\sqrt{\epsilon}B(t)-\sqrt{\epsilon}B(s))^{\frac d2}}\int_{\mathbb R^d} \exp\Big(\bi \frac {|x-y|^2}{4(\sqrt{\epsilon}B(t)-\sqrt{\epsilon}B(s))}\Big)zdy\Big\|_{L^p}\\
&\le C_d \epsilon^{-\frac d2(\frac 12-\frac 1p)}(B(t)-B(s))^{-d(\frac 12-\frac 1p)}\|z\|_{L^{p'}},   
\end{align*}
Thus it holds that for $p'= {2-2\eta}$ and $p=\frac {2-2\eta}{1-2\eta}$ with $\eta\in [0,\frac 12),$
\begin{align*}
&\|X^{\epsilon}(t)\|_{L^p}\\
\le& \|S_B(t)u_0\|_{L^p}+\|\int_0^tS_B(t,s)\log(|X^{\epsilon}|^2)X^{\epsilon}ds\|_{L^p}\\
\le& C_d\epsilon^{-\frac d2(\frac 12-\frac 1p)} B(t)^{-d(\frac 12-\frac 1p)} \|u_0\|_{L^{p'}}
\\
&+C_d\int_0^t\epsilon^{-\frac d2(\frac 12-\frac 1p)} (B(t)-B(s))^{- d(\frac 12-\frac 1p)}\|\log(|X^{\epsilon}|^2)X^{\epsilon}\|_{L^{p'}}ds.
\end{align*}
Note that by the weighted Sobolev inequality, GN inequality and the properties of logarithmic function,  it holds that for $\eta\in (0,1),\alpha>\frac {d\eta}{2-2\eta}$,
\begin{align*}
  \|X^{\epsilon}\|_{L^{2-2\eta}}&\le C\|X^{\epsilon}\|^{1-\frac {d\eta}{2\alpha(1-\eta)}}\|X^{\epsilon}\|_{L_{\alpha}^2}^{\frac {d\eta}{2\alpha(1-\eta)}}  
\end{align*}
and that for $\eta_1,\eta_1'>0$ small enough and $\alpha_1>\frac {d(\eta+\eta_1)}{2-2\eta-2\eta_1}$,
\begin{align*}
\|\log(|X^{\epsilon}|^2)X^{\epsilon}\|_{L^{2-2\eta}}
&\le C (\|X^{\epsilon}\|_{L^{2-2\eta-2\eta_1}}+\|X^{\epsilon}\|_{L^{2-2\eta-2\eta_1'}})\\
&\le C\|u_0\|^{1-\frac {d(\eta+\eta_1)}{2\alpha_1(1-\eta-\eta_1)}}\|X^{\epsilon}\|_{L_{\alpha_1}^2}^{\frac {d(\eta+\eta_1)}{2\alpha_1(1-\eta-\eta_1)}}\\
&+C\|u_0\|^{\frac {d{\eta+\eta_1'}}{2+2\eta+2\eta_1'}}\|\nabla X^{\epsilon}\|^{1-\frac {d{(\eta+\eta_1')}}{2+2\eta+2\eta_1'}}.
\end{align*}
Thanks to the a priori estimate on $X^{\epsilon}$, it holds that 
\begin{align*}
&\|\log(|X^{\epsilon}|^2)X^{\epsilon}\|_{L^{p'}}\\
\le& e^{C(\lambda,\|u_0\|,d,\eta,\eta_1,\eta_1')t}\Big[1+(\epsilon\|u_0\|_{H^1}+\|u_0\|_{L_{\alpha}^2})^{\frac {d(\eta+\eta_1)}{2\alpha_1(1-\eta-\eta_1)}}+\|\nabla u_0\|^{1-\frac {d{(\eta+\eta_1')}}{2+2\eta+2\eta_1'}}\Big].
\end{align*}
Then one may expect that 
\begin{align*}
   \|X^{\epsilon}(t)\|_{L^p}\sim O(\epsilon^{-\frac d2(\frac 12-\frac 1p)+\frac {d(\eta+\eta_1)}{2\alpha(1-\eta-\eta_1)}}),\quad a.s.
\end{align*}
Note that the above asymptotic estimate is not trivial especially in the case that the Sobolev embedding theorem $H^1 \hookrightarrow  L^p$ does not hold.

\section{Appendix}

\textbf{Proof of Theorem \ref{ldp-chi1}}
\begin{proof}

First we show that under the $C([0,T];\mathcal X_1)$-norm, the considered model satisfies the LDP with the same good rate function in section \ref{sec-ldp}. From the arguments in section \ref{sec-ldp}, it suffices to show that the trajectories of \eqref{reg-log-rands} and its skeleton equation are continuous in $\mathcal X_1$.  For simplicity, we only present the detailed proof for the continuity of $L_{u_0}(g)$ in $\mathcal X_1$ since the other argument is similar. 

Define $W_{1,x}^2:=\{z\in H^1|\int_{\mathcal O}(1+|x|^2)(|z(x)|^2+|\nabla z(x)|^2)dx<\infty\}.$
Let us consider a sequence of approximations $u_0^{R'}\in  H^{2}\cap W_{1,x}^2$ to $u_0$ such that 
$\|u_0^{R'}-u_0\|_{\mathcal X_1}\to 0$ as $R'\to + \infty.$
By using the mild formulation of $L_{u_0^{R'}}(h)(t)$ and the Gagliardo--Nirenberg inequality $\|w\|_{L^4}\le C_d\|\nabla w\|^{\frac d4}\|w\|^{1-\frac {d}4}$, we obtain that for $d\le 2,$
\begin{align*}
&\|L_{u_{0}^{R'}}^{\delta}(g)(t)\|_{H^2}\\
\le& \|u_{0}^{R'}\|_{H^2}
+\int_0^t |\lambda|\Big\|S_{g}(t,s)f_{\delta}(|L_{u_{0}^{R'}}^{\delta}(g)(s)|^2)L_{u^{R'}_0}^{\delta}(g)(s)\Big\|_{ H^2}ds\\
&+\int_0^t \alpha_1 \Big\|S_{g}(t,s)L_{u_{0}^{R'}}^{\delta}(g)(s)\Big\|_{ H^2}ds\\
\le& \|u_{0}^{R'}\|_{H^2}+\int_0^t C(\lambda,\delta,\alpha_1)(\|L_{u_{0}^{R'}}^{\delta}(g)(s)\|_{ H^2}+\| L_{u_{0}^{R'}}^{\delta}(g)(s)\|_{ H^2}^{\frac d2}\|L_{u_0^{R'}}^{\delta}(g)(s)\|_{ H^1}^{2-\frac d2})ds.
\end{align*}
As a consequence, the global estimate holds,
\begin{align*}
    \sup_{t\in [0,T]}\|L_{u_{0}^{R'}}^{\delta}(g)(t)\|_{ H^2}\le C(\lambda,\delta,\alpha_1,T,\|u_0^{R'}\|_{H^1}) \|u_0^{R'}\|_{ H^2}.
\end{align*}
Next, we deal with the $W_{1,x}^2$-estimate. By the chain rule and integration by parts, it follows that 
\begin{align*}
  &\partial_t \|x\nabla L_{u_{0}^{R'}}^{\delta}(g)(t)\|^2\\
  \le&   4\Big\<x\nabla L_{u_{0}^{R'}}^{\delta}(g)(t),\bi \Delta L_{u_{0}^{R'}}^{\delta}(g)(t) \Big\>g(t)\\
  &+4|\lambda|\Big\<|x|^2\nabla L_{u_{0}^{R'}}^{\delta}(g)(t) ,f_{\delta}'(| L_{u_{0}^{R'}}^{\delta}(g)(t)|^2) Re\Big(\overline{L_{u_{0}^{R'}}^{\delta}(g)(t)} \nabla L_{u_{0}^{R'}}^{\delta}(g)(t)\Big)L_{u_{0}^{R'}}^{\delta}(g)(t)\Big\>\\
  &-2\alpha_1\|x\nabla L_{u_{0}^{R'}}^{\delta}(g)(t)\|^2.
\end{align*}
The Gronwall's inequality yields that 
\begin{align*}
   \sup_{t\in [0,T]} \|x\nabla L_{u_{0}^{R'}}^{\delta}(g)(t)\|
   &\le C(T,\delta,\alpha_1,\lambda)(\|x\nabla u_{0}^{R'}\|+\|u_{0}^{R'}\|_{H^2}).
\end{align*}

Now we prove the convergence of $L_{u_{0}^{R'}}^{\delta}(g)$ in $\mathcal X_1$ as $R' \to \infty,$ which implies 
that $L_{u_{0}^{R'}}^{\delta}(g)\in \mathcal C([0,T];\mathcal X_1).$ On the one hand, by the unitary property of $S_g$ and the properties of $f_{\delta}$, we have that 
\begin{align*}
 &\|L_{u_{0}^{R'}}^{\delta}(g)-L_{u_{0}}^{\delta}(g)\|_{ H^1}\\
\le& \|S_{g}(t,0)(u_{0}^{R'}-u_0)\|_{H^1}\\
&
+\int_0^{t}|\lambda|\|S_g(t-s)(f_{\delta}(|L_{u_{0}^{R'}}^{\delta}(g)|^2)L_{u_{0}^{R'}}^{\delta}(g)-f_{\delta}(|L_{u_{0}}^{\delta}(g)|^2)L_{u_{0}}^{\delta}(g)\|_{ H^1}ds\\
&+\int_0^{t}|\alpha_1|\|S_g(t-s)(L_{u_{0}^{R'}}^{\delta}(g)-L_{u_{0}}^{\delta}(g))\|_{H^1}ds\\
\le& \|u_{0}^{R'}-u_0\|_{ H^1}
+\int_0^t |\lambda|(|\log(\delta)|+2)\|L_{u_{0}^{R'}}^{\delta}(g)(s)-L_{u_{0}}^{\delta}(g)(s)\|_{ H^1} ds\\
&+\int_0^t |\alpha_1|\|L_{u_{0}^{R'}}^{\delta}(g)(s)-L_{u_{0}}^{\delta}(g)(s)\|_{ H^1}ds. 
\end{align*}
The Gronwall's inequality yields that 
\begin{align}\label{err-h1-r}
    \sup_{t\in [0,T]}\|L_{u_{0}^{R'}}^{\delta}(g)(t)-L_{u_{0}}^{\delta}(g)(t)\|_{ H^1}&\le 
   C(|\lambda|,\delta,T,\alpha_1) \|u_{0}^{R'}-u_0\|_{ H^1}\to 0,\; \text{as} \; R' \to \infty.
\end{align}
On the other hand, applying the chain rule and integration by parts, as well as Young's inequality, one can derive
\begin{align*}
  &\|x(L_{u_{0}^{R'}}^{\delta}(g)(t)-L_{u_{0}}^{\delta}(g)(t))\|^2
  \\
  \le& \|x(u_{0}^{R'}-u_{0})\|^2-2\alpha_1 \int_0^t\|x(L_{u_{0}^{R'}}^{\delta}(g)(s)-L_{u_{0}}^{\delta}(g)(s))\|^2 ds
  \\
  &+\int_0^t 4\<x(L_{u_{0}^{R'}}^{\delta}(g)(s)-L_{u_{0}}^{\delta}(g)(s)), \bi \nabla (L_{u_{0}^{R'}}^{\delta}(g)(s)-L_{u_{0}}^{\delta}(g)(s)) \>h(s)ds\\
  \le& \|x(u_{0}^{R'}-u_0)\|^2+\int_0^t2\| \nabla (L_{u_{0}^{R'}}^{\delta}(g)(s)-L_{u_{0}}^{\delta}(g)) \| h^2(s)ds
  \\
  &+\int_0^t (2-2\alpha_1)\|x(L_{u_{0}^{R'}}^{\delta}(g)-L_{u_{0}}^{\delta}(g))\|^2ds.
\end{align*}
According to 
\eqref{err-h1-r} and Gronwall's inequality, we obtain that 
\begin{align}\label{weight-err-r}
  &\sup_{t\in [0,T]}\|L_{u_{0}^{R'}}^{\delta}(g)(t)-L_{u_0}^{\delta}(g)(t)\|_{L_1^2}\\\nonumber
  \le& C(\alpha_1,T)\|u_0^{R'}-u_0\|_{L_1^2}
  +C(|\lambda|,\delta,T,\alpha_1) \|u_0^{R'}-u_0\|_{ H^1}\to 0, \;\text{as} \; R'\to \infty.
\end{align}
Moreover, one can verify that $u^{\epsilon,u_0^{R'},\delta}$ is an exponentially 
good approximation of $u^{\epsilon,\delta}$ and thus $I_{u_0}^{\delta}$ is also a good rate function under $\mathcal C([0,T]; \mathcal X_1).$

Now we are able to prove the desired result.
Since $I^{W}$ is a good rate function,  $K_T^{u_0}(\widetilde a)$ is compact set of $\mathcal C([0,T];\mathcal X_1)$ for any $\widetilde a$.
Define 
$$A_{\widetilde a}^{u_0}:=\{v\in \mathcal C([0,T;\mathbb \mathcal X_1])| d_{\mathcal C([0,T])}(v,K_T^{u_0}(\widetilde a))\ge \gamma\}.$$
Choosing $g$ such that $I^W(g)<\widetilde a,$ it follows that 
\begin{align*}
    &\mathbb P(u^{\epsilon,\delta,u_0}\in A_{\widetilde a}^{u_0})
    \le \mathbb P\Big( \|u^{\epsilon,\delta,u_0}-L_{u_0}^{\delta}(g)\|_{\mathcal C([0,T];\mathcal X_1)}\ge \gamma \Big).
\end{align*}
Then by the LDP, there exists $\epsilon_0>0$ such that for every $\epsilon \in (0,\epsilon_0),$
\begin{align*}
    \epsilon \log \mathbb P(u^{\epsilon,u_0,\delta}\in A_{\widetilde  a}^{(u_0)})\le -(\widetilde a-\kappa),
\end{align*}
which implies the upper bound. 
Next, we consider the lower bound.
Due to the continuity of $L_{u_0}^{\delta}(\cdot)$ and the compactness of $C_a$, for any $\|u_0\|_{\mathcal X_1}\le \rho$ and $w\in K_T^{u_0,\delta}(a)$, there exists $g$ such that $w=L_{u_0}^{\delta}(g)$ and $I^{\delta}_{u_0}(w)=I^W(g).$
By the LDP and the fact that $B_{\gamma}^w=\{v|\|v-w\|_{C([0,T];\mathcal X_1)}<\gamma\}$ is an open set, there exists $\epsilon_0'>0$ such that for every $\epsilon\in(0,\epsilon_0'),$ 
\begin{align*}
 \epsilon \log \mathbb P(\|u^{\epsilon,u_0,\delta}-w\|_{\mathcal C([0,T];\mathcal X_1)}<\gamma)&\ge  -\inf_{v\in B_{\gamma}^w} I^{u_0,\delta}(v)-\kappa\\
 &\ge -I^{u_0}(w)-\kappa.
\end{align*}
We complete the proof. 
\end{proof}

\textbf{Proof of Proposition \ref{tm-exit}}

\begin{proof}
Let us first prove the upper bound estimate \eqref{tm-exit3}-\eqref{tm-exit4}. 
Fix $\kappa>0$ small enough and choose $g$ and $T_1'$ such that $L_{0}^{\delta}(g)(T_1')\in (\overline{D})^{c}$ and 
$$I_{0}^{\delta, T_1}(L_0^{\delta}(g)(T_1'))=\frac 12 \|h\|^2_{L^2([0,T_1'];\mathbb R)}\le \overline {M^{\delta}}+\frac {\kappa}6.$$
Let $d_0$ denote the positive distance between $L_{0}^{\delta}(g)(T_1')$ and $\overline{D}.$
Since $f_{\delta}$ is Lipschitz for a fixed $\delta$, there exist a small ball $B_{\rho}^0\subset D$ such that if $u_0$ belongs to $B_{\rho}^0,$ 
\begin{align*}
 &\|L_{u_0}^{\delta}(g)-L_{0}^{\delta}(g)\|_{C([0,T_1'];\mathcal X_1)} < \frac {d_0}2.
\end{align*}
By the LDP in Theorem \ref{ldp-chi1} and triangle inequality, there exists $\epsilon_1'$ such that for every $\epsilon\in (0,\epsilon_1')$,
\begin{align*}
    \mathbb P (\tau^{\epsilon,\delta,u_0}<T_1')
    &\ge \mathbb P(\|u^{\epsilon,\delta,u_0}-L_{0}^{\delta}(g)\|_{C([0,T_1'];\mathcal X_1)}<d_0)\\
    &\ge \mathbb P(\|u^{\epsilon,\delta,u_0}-L_{u_0}^{\delta}(g)\|_{C([0,T_1'];\mathcal X_1)}<\frac {d_0}2)\\
    &\ge \exp\Big(-\frac {I_{u_0}^{\delta,T_1'}+\frac {\kappa}6}\epsilon\Big).
\end{align*}
From Lemma \ref{lm-1}, there exist $T_2'$ and $\epsilon_2'$ such that for $\epsilon\in (0,\epsilon_2')$,
$$\inf_{u_0\in D}\mathbb P(\sigma_{\rho}^{\epsilon,\delta,u_0}\le T_2)\ge \frac 12.$$ Applying the Markov property, we get that for $\epsilon\in(0,\epsilon_1'\wedge \epsilon_2'),$
\begin{align*}
\inf_{u_0\in D}\mathbb P(\tau^{\epsilon,\delta,u_0}\le T_1'+T_2')
&\ge \inf_{u_0\in D} \mathbb P(\sigma_{\rho}^{\epsilon,\delta,u_0}\le T_2')\inf_{u_0\in B_{\rho}^0} \mathbb P(\tau^{\epsilon,\delta,u_0}\le T_1')\\
&\ge \exp\Big(-\frac {I_{u_0}^{\delta,T_1'}+\frac {\kappa}3}\epsilon\Big), 
\end{align*}
where we have used the fact that $\epsilon$ is small enough such that $\frac 12\ge e^{-\frac {\kappa}{6 \epsilon}}.$
Then for any $k\ge 1,$ using the property of conditional probability, it holds that 
\begin{align*}
 &\mathbb P(\tau^{\epsilon,\delta,u_0}>(k+1)(T_1'+T_2'))\\
=&\Big[1-\mathbb P\Big(\tau^{\epsilon,\delta,u_0}\le (k+1)(T_1'+T_2')\Big|\tau^{\epsilon,\delta,u_0}>k(T_1'+T_2')\Big)\Big]\\
 &\quad \times \mathbb P(\tau^{\epsilon,\delta,u_0}>k(T_1'+T_2'))\\
 \le& (1-\inf_{u_0\in D}\mathbb P(\tau^{\epsilon,\delta,u_0}\le T_1'+T_2'))\mathbb P(\tau^{\epsilon,\delta,u_0}>k(T_1'+T_2'))\\
 \le& (1-\inf_{u_0\in D}\mathbb P(\tau^{\epsilon,\delta,u_0}\le T_1'+T_2'))^k.
\end{align*}
Notice that $I_{u_0}^{\delta,T_1'}(L_{u_0}^{\delta}(g))=I_{0}^{\delta, T_1'}(L_{0}^{\delta}(g))=\frac 1  2 \|h\|_{L^2([0,T_1';\mathbb R])}^2$. Thus, we also have 
\begin{align*}
\sup_{u_0\in D}\E[\tau^{\epsilon,\delta,u_0}]
&=\sup_{u_0\in D}\int_0^{+\infty} \mathbb P(\tau^{\epsilon,\delta,u_0}>t)dt\\
&\le (T_1'+T_2')\sum_{k=0}^{\infty}\sup_{u_0 \in D}\mathbb P(\tau^{\epsilon,\delta,u_0}>k(T_1'+T_2'))\\
&\le \frac {(T_1'+T_2')}{1-\inf\limits_{u_0\in D}\mathbb P(\tau^{\epsilon,\delta,u_0}\le T_1'+T_2')}\\
&\le (T_1'+T_2')\exp\Big(\frac {\overline{M^{\delta}}+\frac {\kappa}{2}}\epsilon\Big).
\end{align*}
One can take $\epsilon$ small enough such that 
\begin{align}
   \sup_{u_0\in D}\E [\tau^{\epsilon,\delta,u_0}] 
   &\le \exp\Big(\frac {\overline{M^{\delta}}+\frac {2\kappa}{3}}\epsilon\Big),
\end{align}
which implies \eqref{tm-exit4}.
The Chebyshev inequality yields that 
\begin{align*}
    \sup_{u_0\in D}\mathbb P(\tau^{\epsilon,\delta,u_0}\ge e^{\frac {\overline{M^{\delta}+\kappa}}{\epsilon}})
    \le e^{-\frac {\overline{M^{\delta}}+\kappa}{\epsilon}}\sup_{u_0\in D}\E[\tau^{\epsilon,\delta,u_0}]
    &\le e^{-\frac {\kappa}{3\epsilon}}.
\end{align*}
As a consequence, \eqref{tm-exit3} follows.

Below we are in a position to deal with lower bound estimate of $\tau^{\epsilon,\delta,u_0}.$
By Lemma \ref{lm-0}, we can choose $\rho>0$ small enough such that  $\underline{M^{\delta}}-\frac {\kappa}4\le M^{\delta}_{\rho}$ and $B_{2\rho}^0\subset D$, where $0<\frac {\kappa}4<\inf\limits_{\delta> 0}\underline{M^{\delta}}.$ Similar to \cite{DZ98lp}, one can define two sequences of stopping times,
\begin{align*}
\tau_{k}&=\inf \{t\ge \theta_k | u^{\epsilon,\delta,u_0}(t)\in B_{\rho}^0\cup D^c\},\\
\theta_{k+1}&=\inf \{t>\tau_k |u^{\epsilon,\delta,u_0}(t)\in S_{2\rho}^0\},
\end{align*}
where $\theta_0=0, k\in \mathbb N$ and $\theta_{k+1}=+\infty$ if $u^{\epsilon,\delta,u_0}(\tau_k)\in \partial D.$

Choosing $T_3$ in Lemma \ref{lm-4} that satisfies $L=\underline{M^{\delta}}-\frac {3\kappa}4$, there exists $\epsilon$ small enough such that for all $k\ge 1$ and $u_0\in D$,
\begin{align*}
    \mathbb P(\theta_k-\tau_{k-1}\le T_3)\le e^{-\frac {\underline{M^{\delta}}-\frac {3\kappa}4} \epsilon}.
\end{align*}
Note that the escapes before $mT_3$ with $m\in \mathbb N^+$ occur in three cases, that is, Case 1, the escape occurs without passing $B_{\rho}^0$; Case 2, the trajectory of $u^{\epsilon,u_0,\delta}$ crosses $S_{2\rho}^0$ with $k$ times and then escapes at $\tau_k$; Case 3: the escape occurs after $\tau_m$ which implies that there exists at least the length of one interval $[\tau_{k-1},\theta_{k}]$ is smaller than $T_3$.

As a consequence, for $u_0\in D$,
{\small
\begin{align}\label{iter-exp}
\mathbb P(\tau^{\epsilon,\delta,u_0}\le m T_3)
&\le \mathbb P(\tau^{\epsilon,\delta,u_0}=\tau_0)
+\sum_{k=1}^m\mathbb P(\tau^{\epsilon,\delta,u_0}=\tau^k)+\sum_{k=1}^m \mathbb P(\theta_{k}-\tau_{k-1}<T_3)\\\nonumber
&\le \mathbb P(\tau^{\epsilon,\delta,u_0}=\tau_0)
+\sum_{k=1}^m\mathbb P(\tau^{\epsilon,\delta,u_0}=\tau^k)+me^{-\frac {\underline{M^{\delta}}-\frac {3\kappa}4} \epsilon}.
\end{align}}
Below we will estimate $\mathbb P(\tau^{\epsilon,\delta,u_0}=\tau^k), k\ge 1,$ by using the fact that $$\mathbb P(\tau^{\epsilon,\delta,u_0}=\tau_k)
\le \mathbb P(\tau^{\epsilon,\delta,u_0}\le T_4,\tau^{\epsilon,\delta,u_0}=\tau_k)+\sup_{y\in S_{2\rho}^0}\mathbb P(\sigma_{\rho}^{\epsilon,\delta,y}> T_4)
$$
for all $T_4>0$ with $y=u^{\epsilon,\delta,u_0}(\theta_{k-1})\in S_{2\rho}.$
On the one hand, choosing $T_4$ as the time in Lemma \ref{lm-2} and $L=\underline{M^{\delta}}-\frac {3\kappa}4$, we obtain that 
\begin{align*}
    \mathbb P(\sigma_{\rho}^{\epsilon,\delta,y}> T_4)&\le e^{-\frac {\underline{M^{\delta}}-\frac {3\kappa}4}{\epsilon}}.
\end{align*}
On the other hand, using the LDP in Theorem \ref{ldp-chi1},
there exists $\epsilon$ small enough such that for $u_1\in B_{\rho}^0$,
\begin{align*}
     \mathbb P(\tau^{\epsilon,\delta,u_1}\le T_4) &\le   \mathbb P\Big(d_{\mathcal C([0,T_4];\mathcal X_1)}\Big(u^{\epsilon,\delta,u_1}, K_{T_4}^{u_1}(M_{\rho}^{\delta}-\frac {\kappa}4)\Big)\ge \rho \Big)\\
    &\le e^{-\frac {M_{\rho}^{\delta}-\frac {\kappa}2}{\epsilon}}\le  e^{-\frac {\underline{M^{\delta}}-\frac {3\kappa}4}{\epsilon}}.
\end{align*}
As a consequence, we have that 
{\small
\begin{align*}
\mathbb P(\tau^{\epsilon,u_0,\delta}\le T_4,\tau^{\epsilon,u_0,\delta}=\tau_k)&\le 
\mathbb P(\tau^{\epsilon,\delta,u^{\epsilon,\delta,u_0}(\tau^{k-1})}\le T_4, \tau_{k}-\tau_{k-1}\le T_4)\le  e^{-\frac {\underline{M^{\delta}}-\frac {3\kappa}4}{\epsilon}}.
\end{align*} 
}
Combining the above estimates and \eqref{iter-exp}, 
it holds that for $\epsilon$ small enough, 
\begin{align*}
    \mathbb P(\tau^{\epsilon,\delta,u_0}\le mT_3)
    &\le \mathbb P(\tau^{\epsilon,\delta,u_0}=\tau_0)+3me^{-\frac {\underline{M^{\delta}}-\frac {3\kappa}4}{\epsilon}}\\
    &\le \mathbb P(u^{\epsilon,\delta,u_0}(\sigma^{\epsilon,\delta,u_0}_{\rho})\in \partial D)+3me^{-\frac {\underline{M^{\delta}}-\frac {3\kappa}4}{\epsilon}}.
\end{align*}
Using Lemma \ref{lm-2}, taking $m=[\frac 1{T_3}\exp(\frac {\underline{M^{\delta}}-\kappa}{\epsilon})]$, we obtain \eqref{tm-exit1}.
Applying the Chebyshev's inequality, the desired lower bound on $\E[\tau^{\epsilon,\delta,u_0}]$ follows.

\end{proof}

\bibliographystyle{plain}
\bibliography{bib}

\end{document}